   \documentclass[final,leqno,onefignum,onetabnum]{siamart1116}
  \numberwithin{equation}{section}
   \numberwithin{theorem}{section}
 
\usepackage{bm}
\usepackage{graphicx}
\usepackage{amsmath}
\usepackage{amsfonts,euscript,eufrak,latexsym,amssymb,stmaryrd}
\usepackage{epstopdf} 
\usepackage{algorithm}
\usepackage{algpseudocode}
\usepackage{cleveref}

\DeclareFontFamily{OT1}{pzc}{}
\DeclareFontShape{OT1}{pzc}{m}{it}{<-> s * [1.10] pzcmi7t}{}
\DeclareMathAlphabet{\mathpzc}{OT1}{pzc}{m}{it}

\newcommand\oscf{{\rm osc}(f)}
\newcommand\osckf{{\rm osc}_{K}(f)}

\newcommand\B{{\mathcal B}}

\newcommand\Ch{{\mathcal C}_h}

\newcommand\Oc{\Omega^{c}_{h}}

\newcommand\Ohc{\Omega^{c}_{h}}

\newcommand\Ohck{\Omega^{c}_{h,k+1}}

\newcommand\Eh{{\mathcal E}_h}

\newcommand\PK{\Pi_h}

\newcommand\h{\mathpzc{h}}

\newcommand\ve{w_{E}}

\newsiamremark{remark}{Remark}
\begin{document}

\title{Mixed and stabilized finite element methods for the obstacle problem\thanks{Financial support from Tekes (Decisions nr. 40205/12 and 3305/31/2015), Portuguese Science Foundation (FCOMP-01-0124-FEDER-029408) and Finnish Cultural Foundation is gratefully acknowledged.
}}

\author{Tom Gustafsson\thanks{Department of Mathematics and Systems Analysis,
Aalto University, P.O. Box 11100,  \hfill \break00076 Aalto, Finland
e-mail: 
(\email{tom.gustafsson@aalto.fi}). }  \and Rolf Stenberg\thanks{Department of Mathematics and Systems Analysis,
Aalto University, P.O. Box 11100,  \hfill \break00076 Aalto, Finland
e-mail: 
(\email{rolf.stenberg@aalto.fi}). }  
 \and {Juha Videman}\thanks{CAMGSD/Departamento de Matem\'atica, Instituto Superior T\'ecnico, Universidade   de Lisboa, Av. Rovisco Pais 1, 1049-001 Lisboa, Portugal (\email{jvideman@math.tecnico.ulisboa.pt}). }  
}

\maketitle
\begin{abstract}
    We discretize the Lagrange multiplier formulation of the obstacle problem by mixed and stabilized finite element methods. A priori and a posteriori error estimates are
    derived and numerically verified.
\end{abstract}




\section{Introduction}
\label{intro}

In its classical form, the obstacle problem is an archtype example of a variational inequality~\cite{LS67}. The problem corresponds to finding the equilibrium position of an elastic membrane constrained to lie above a rigid obstacle.
Other examples of obstacle-type problems are found, e.g., in lubrication theory~\cite{R87}, in flows through porous media~\cite{Fri82}, in control theory~\cite{PSU12} and in financial mathematics~\cite{Eva14}.

Discretization of the primal variational formulation of the obstacle problem by the finite element method has been extensively studied since the 1970's. Error estimates for the membrane displacement in the $H^1$-norm have been obtained, e.g., by Falk~\cite{F74}, Mosco--Strang~\cite{MS73} and Brezzi--Hager--Raviart~\cite{BHR77}. For an overview of the early progress on the subject and the respective references, see the monograph of Glowinski~\cite{G84}.

Instead of focusing on the primal formulation, we study an alternative variational formulation based on the method of Lagrange multipliers~\cite{B73,Br74,BHR78,Has81}. The Lagrange multiplier formulation introduces an additional physically relevant unknown, the reaction force between the membrane and the obstacle, which in itself can be a useful tool---especially in the context of contact mechanics, cf.~Hlav{\'a}{\v{c}}ek et al.~\cite{HHNL} or, more recently, Wohlmuth~\cite{W11}.
Furthermore, the alternative formulation leads naturally  to an effective solution strategy based on the semismooth Newton method~\cite{HIK03,U11} and  provides also a straightforward justification for the related Nitsche-type method that follows from the local elimination of the Lagrange multiplier in the stabilized discrete problem~\cite{S75}. 

A priori error analysis for  finite element methods based on a Lagrange multiplier formulation of the obstacle problem has been performed by Haslinger et al.~\cite{HHN}, Weiss--Wohlmuth~\cite{WW10} and  Schr\"oder et al.~\cite{SCH11,BS15}.  A posteriori error estimates were derived, e.g., in B{\"u}rg--Schr{\"o}der~\cite{BUSCH15} or Banz--Stephan~\cite{BASTEP} and, using a similar Lagrange multiplier formulation, in Veeser~\cite{V01}, Braess~\cite{B05} and Gudi--Porwal~\cite{GP15}.

The purpose of this paper is to readdress the Lagrange
  multiplier formulation of the obstacle problem.   
  We consider two  methods. The first is a mixed method in which the stability is achieved by adding bubble degrees of freedom  to the displacement. The second approach is  a residual-based stabilized method similar to the methods  used successfully for the Stokes problem \cite{Hughes-Franca,FS91}. The latter technique has been applied to variational inequalities arising from contact problems in Hild--Renard~\cite{HR10}.
  
Until our recent article on the Stokes problem \cite{SV15}, error estimates on stabilized methods were always built upon the assumption that the exact solution is regular enough, with the additional regularity assumptions arising from the stabilizing terms.  In \cite{SV15}, we realized that these extra regularity requirements can be dropped and only the loading term needs to be in $L^2$. In this work, we extend these ideas to variational inequalities. We emphasize that the improved error estimates can only be established if the discrete stability bound is proven in correct norms, i.e.~in the norms of the continuous problem and not in mesh dependent norms
 as, for example, in Hild--Renard~\cite{HR10}  where the authors need to assume that the solution to the Signorini problem is in $H^s$, with $s>3/2$. 
  
 We perform the analysis in a unified manner. First, we prove a stability result for the continuous problem in proper norms. This estimate becomes useful in deriving the a posteriori estimates. As for the discretizations, we start by proving  stability with respect to a mesh dependent norm. This discrete stability result implies stability in the continuous norms and yields quasi optimal a priori estimates without  additional regularity assumptions. For the stabilized methods, we use a technique first suggested by Gudi~\cite{Gudi}.  
 
 The stabilized formulation of the classical Babu$\check{\rm s}$ka's method of Lagrange multipliers for approximating Dirichlet boundary conditions is known to be closely related to Nitsche's method \cite{NIT}. This connection has been used for the contact problem in Hild--Renard~\cite{HR10} and Chouly--Hild~\cite{CH13}, and for the obstacle problem in \cite{BHLS}. We show that a similar relationship holds here as well and  observe that for the lowest order method it leads to the penalty formulation. 
 
 We end the paper by reporting on extensive numerical computations.

\section{Problem definition}
\label{sec:1}

Consider finding the equilibrium position $u=u(x,y)$ of an elastic, homogeneous membrane constrained to lie above an obstacle represented by $g=g(x,y)$. The membrane is loaded by a  normal force $f=f(x,y)$ and its boundary is held fixed.  The  problem  can be formulated as
\begin{equation}
    \label{strong}
    \begin{aligned}
        - \Delta u-f &\geq 0   &&{\rm in}~\Omega,\\ 
                 u-g &\geq 0   &&{\rm in}~\Omega,\\  
        (u-g)(\Delta u + f) &= 0 &&{\rm in}~\Omega,\\
                   u &=0       &&{\rm on}~\partial \Omega, 
    \end{aligned}
\end{equation}
where  $\Omega\subset\mathbb{R}^2$ is  a smooth bounded domain or a convex polygon, and  where  we assume that $g= 0 $ at $\partial\Omega $. Let $V=H^1_{0}(\Omega)$.
The problem \eqref{strong} can be recast as the following variational inequality:
find $u \in \mathcal{V}$ such that
\begin{equation}
    \big(\nabla u, \nabla (v-u)\big) \geq (f,v-u) \quad \forall v\in \mathcal{V} , 
\label{primalvi}
\end{equation}
where 
\begin{equation}\mathcal{V}=\{v\in  V :  v \geq g~{\rm a.e.~in }\ \Omega\}.
\end{equation}
It is well known that, given $f\in L^{2}(\Omega)$ and $g\in H^1(\Omega)\cap C(\overline{\Omega})$, problem \eqref{primalvi} admits a unique solution $u\in \mathcal{V}$, cf.~Lions--Stampacchia~\cite{LS67}, or Kinderlehrer--Stampacchia~\cite{KS80}.  
Given that the second derivatives of  $u$ have  a jump across the free boundary separating the contact region from the region free of contact, one cannot expect the solution to be more regular than $C^{1,1}(\Omega)$, cf.~Caffarelli~\cite{Caf97}. However, the second derivatives are bounded if the data is smooth~\cite{Fri82}. In particular, if  $f\in L^2(\Omega)$ and $g\in H^2(\Omega)$, the solution $u\in \mathcal{V}\cap H^{2}(\Omega)$, cf.~Brezis--Stampacchia~\cite{BS68}.

Introducing a non-negative Lagrange multiplier function $\lambda:\Omega\rightarrow \mathbb{R}$, we can rewrite the obstacle problem  as
\begin{equation}
    \label{stronglagr}
    \begin{aligned} 
        - \Delta u-\lambda&= f      &&{\rm in}~\Omega,\\
                       u-g&\geq 0   &&{\rm in}~\Omega,\\  
                   \lambda&\geq 0   &&{\rm in}~\Omega,\\  
           (u-g)\, \lambda&= 0      &&{\rm in}~\Omega,\\  
                         u&=0       &&{\rm on}~\partial \Omega.
    \end{aligned}
\end{equation}
The Lagrange multiplier is in the dual space to $H_0^1(\Omega)$, i.e.
$$  Q =H^{-1}(\Omega),$$
with the norm
\begin{equation}\label{negnorm}
    \|\xi\|_{-1}=\sup_{v \in V}\frac{\langle  v, \xi\rangle   }{\Vert   v\Vert_{1}},
\end{equation}
where $\langle \cdot,\cdot \rangle : V \times Q\rightarrow \mathbb{R}$ denotes the duality pairing.
 
The corresponding variational formulation becomes: 
find $(u,\lambda)\in V\times \Lambda$ such that
\begin{equation}
    \label{dualvi}
    \begin{aligned}
        (\nabla u, \nabla v ) -\langle v, \lambda\rangle &=  (f,v) &&\forall v\in V,  \\
        \langle u-g,\mu-\lambda\rangle &\geq 0 &&\forall \mu \in \Lambda,
    \end{aligned}
\end{equation}
where 
$$\Lambda=\{\mu\in Q : \langle v, \mu \rangle \geq 0 ~~\forall v \in V,~v \geq 0 ~{\rm a.e.~in }~\Omega \}.$$
Existence of a unique solution  $(u,\lambda)\in V\times \Lambda$ to the mixed problem \eqref{dualvi} and equivalence between formulations \eqref{primalvi} and \eqref{dualvi} has been proven, e.g., in Haslinger et al.~\cite{HHN}.

Let $U=V\times Q$ and define the bilinear form $\mathcal{B}:U\times U\rightarrow \mathbb{R}$ and the linear form $\mathcal{L}:U\rightarrow \mathbb{R}$ through 
\begin{align*}
    \mathcal{B}(w,\xi;v,\mu)&=(\nabla w,\nabla v)-\langle v,\xi \rangle-\langle w,\mu\rangle\, , \\
    \mathcal{L}(v,\mu)&=(f,v)-\langle g,\mu\rangle\, .
\end{align*}
 Problem \eqref{dualvi} can now be written in a compact way as:
find $(u,\lambda)\in V\times \Lambda$ such that
\begin{equation}\begin{array}{rl}
\mathcal{B}(u,\lambda;v,\mu-\lambda) \leq \mathcal{L}(v,\mu-\lambda) \qquad \forall(v,\mu)\in V\times \Lambda \, .
\label{bform}
\end{array}
\end{equation} 

Our analysis  is built upon the following stability condition.
Note that we often write $a \gtrsim b$ (or $a \lesssim b$) when $a \geq C b$ (or $a \leq C b$) for some positive constant $C$ independent of the finite element mesh.
\begin{theorem} \label{contstab} For all $(v,\xi)\in V\times Q$ there exists $w\in V$ such that 
\begin{equation}
 \B(v,\xi;w,-\xi)\gtrsim   \big(   \Vert v  \Vert_{1} + \Vert \xi  \Vert_{-1}  \big)^{2}
 \end{equation}
  and 
   \begin{equation}
\Vert w\Vert_{1} \lesssim \Vert v  \Vert_{1} + \Vert \xi  \Vert_{-1} .
 \end{equation}
\end{theorem}
\begin{proof}
Suppose the pair $(v,\xi)\in V \times Q$ is given and
let $w=v-q$, where $q\in V$ satisfies
\begin{equation}
\label{thm21:dualproblem}
    (\nabla q, \nabla z)+(q,z) = \langle z, \xi \rangle \quad \forall z \in V.
\end{equation}
Choosing the test function $z=q$, gives
\begin{equation}
    \langle q, \xi \rangle = (\nabla q, \nabla q) + (q,q) = \|q\|_1^2 ,
\end{equation}
and hence we obtain
\begin{equation}
    \label{thm21:result1}
   \|q\|_1 = \frac{\langle \xi, q \rangle}{\|q\|_1} \leq \sup_{z \in V}\frac{\langle \xi, z \rangle}{\|z\|_1} =\|\xi\|_{-1}.
\end{equation}
The norm of $\xi$ can be bounded from above using \eqref{thm21:dualproblem} and the Cauchy--Schwarz inequality as follows:
\begin{equation}
\label{thm21:result2}
    \begin{aligned}
        \|\xi\|_{-1} &= \sup_{z\in V} \frac{\langle z, \xi \rangle}{\|z\|_1} = \sup_{z\in V} \frac{(\nabla q, \nabla z)+(q,z)}{\|z\|_1} \leq \|q\|_1.
    \end{aligned}
\end{equation}
This implies that $\|q\|_1 = \|\xi\|_{-1}$.
Using now the results \eqref{thm21:result1} and \eqref{thm21:result2} and Poincar\'{e}'s and Cauchy--Schwarz inequalities, we conclude that
\begin{align*}
        \mathcal{B}(v,\xi;w,-\xi) &= (\nabla v,\nabla w)+\langle v-w, \xi \rangle \\
                                  &= (\nabla v,\nabla v)-(\nabla v, \nabla q)+\langle q, \xi \rangle \\
                                  &\geq \|\nabla v\|_0^2-\|\nabla v\|_0 \|\nabla q\|_0+\|q\|_1^2 \\
                                  &\gtrsim \left( \|v\|_1 + \|\xi\|_{-1}\right)^2.
\end{align*}
Finally, from the triangle inequality it follows that
\begin{align*}
        \|w\|_{1} &= \|v-q\|_1 \leq \|v\|_1+\|q\|_1 = \|v\|_1+\|\xi\|_{-1}.
\end{align*}
\end{proof}

We will consider finite element spaces based on a conforming shape-regular  triangulation $\Ch$ of $\Omega$, which we henceforth assume to be polygonal. By $\Eh$ we denote the interior edges  of $\Omega$.   
The finite element subspaces are 
\[
V_h\subset V, \quad  Q_h\subset Q.\ \]
Moreover, we define
\[
 \Lambda_h=\{\mu_h\in Q_h : \mu_h  \geq 0 \ {\rm in} \ \Omega\}\subset \Lambda .
\]

\begin{remark}
 When $Q_h$ are piecewise polynomials of degree two or higher, the condition $\mu_h\geq 0$ becomes difficult to satisfy in practice. In that case, one option would be to implement this condition in discrete points, as is done in \cite{BS15}. We do not pursue this path however since it would lead to a nonconforming method with $\Lambda_h\not\subset \Lambda$ and require a separate analysis.
 \end{remark}
We first consider methods corresponding to the continuous problem \eqref{bform}.

\section{Mixed methods}
\label{sec:2}
 
The mixed finite element method for problem \eqref{bform} reads as follows.

\textsc{The mixed method. }{\em
Find $(u_h,\lambda_h)\in V_h\times \Lambda_h$ such that
\begin{equation}\label{mixedm} \begin{array}{rl}
\B(u_h,\lambda_h;v_h,\mu_h-\lambda_h) \leq \mathcal{L}(v_h,\mu_h-\lambda_h) \qquad \forall(v_h,\mu_h)\in V_h\times \Lambda_h .
\end{array}
\end{equation}}

For this class of methods, the finite element spaces have to satisfy the  ``Babu$\check{\rm s}$ka--Brezzi'' condition
\begin{equation}\label{femstab}
\sup_{v_{h}\in V_{h}}\frac{\langle  v_{h}, \xi_{h}\rangle   }{\Vert v_{h}\Vert_{1}}\gtrsim \Vert \xi_{h}\Vert _{-1} \quad \forall \xi_{h}\in Q _{h}. \end{equation}
The Babu$\check{\rm s}$ka--Brezzi condition implies the following discrete stability estimate.
\begin{theorem} If  condition \eqref{femstab} is valid, then  
for all $(v_{h},\xi_{h})\in V_{h}\times Q_{h}$, there exists $w_{h}\in V_{h}$, such that 
\begin{equation}
 \B(v_{h},\xi_{h};w_{h},-\xi_{h}) \gtrsim\big(   \Vert v _{h}\Vert_{1} + \Vert \xi_{h} \Vert_{-1}  \big)^{2}
\end{equation}
and
\begin{equation}
  \Vert w _{h}\Vert_{1}\lesssim \Vert v _{h}\Vert_{1}+ \Vert \xi _{h}\Vert_{-1}.
\end{equation}
\label{dsmixed}
\end{theorem}

\begin{proof}
    Let $w_h = v_h-q_h$, where $q_h \in V_h$ is such that
    \[
        (\nabla q_h, \nabla z_h)+(q_h,z_h) = \langle z_h, \xi_h  \rangle \quad \forall z_h \in V_h.
    \]
   By condition \eqref{femstab} and the Cauchy--Schwarz inequality we have
    \begin{align*}
        \|\xi_h\|_{-1} &\lesssim \sup_{z_h \in V_h} \frac{\langle z_h, \xi_h \rangle}{\|z_h\|_1} = \sup_{z_h \in V_h} \frac{(\nabla q_h, \nabla z_h)+(q_h,z_h)}{\|z_h\|_1} \leq \|q_h\|_1.
    \end{align*}
    Similarly as in the proof of Theorem~\ref{contstab}, we get
    \begin{align*}
        \mathcal{B}(v_h,\xi_h; w_h,-\xi_h) &= (\nabla v_h, \nabla w_h) + \langle \xi_h, q_h \rangle \\
        &= \|\nabla v_h\|_0^2 -(\nabla v_h, \nabla q_h)+\|q_h\|_1^2 \\
        &\gtrsim \|\nabla v_h\|_0^2 + \|q_h\|_1^2 \\
        &\gtrsim \left(\|v_h\|_1 + \|\xi_h\|_{-1}\right)^2.
    \end{align*}
    Finally,
    \begin{align*}
      \|w_h\|_1 &= \|v_h - q_h\|_1 \leq \|v_h\|_1 + \|q_h\|_1 \lesssim \|v_h\|_1 + \|\xi_h\|_{-1}. 
    \end{align*}
\end{proof}

We will use the technique of bubble functions to define a family of stable finite element pairs.
 With $b_{K}\in P_{3}(K)\cap H^{1}_{0}(K)$ we denote  the bubble function scaled to have a maximum value of one and define
 \begin{equation}
B_{l+1}(K)=\{ z \in H^1_0(K) : z=b_Kw, \ w\in  \widetilde P_{l-2}(K) \},
\end{equation}
where $\widetilde P_{l-2}(K) $ denotes the space of homogeneous polynomials of degree $l-2$. 
Let $k\geq 1$ be the degree of the finite element spaces defined by
\begin{equation}\label{Vh} V_h = \begin{cases}
    \{ v_h\in V : v_h\vert _K \in P_1(K)\oplus B_{3}(K) \ \forall K \in \Ch \}   &\mbox{ for } k=1,\\
    \{ v_h\in V : v_h\vert _K \in P_k(K)\oplus B_{k+1}(K) \ \forall K \in \Ch \} &\mbox{ for } k\geq 2, 
\end{cases}
\end{equation}
and let
\begin{equation} 
Q_h = \begin{cases}\label{Lh}
    \{ \xi_h \in Q : \xi_h\vert_K \in P_{0} (K) \ \forall K \in \Ch \} &\mbox{ for } k=1,\\
    \{ \xi_h \in Q : \xi_h\vert_K \in P_{k-2} (K) \ \forall K \in \Ch \}  &\mbox{ for } k\geq 2.
\end{cases}
\end{equation}
Note that the approximation orders of the finite element spaces are balanced, i.e.
\begin{equation}
\label{balance}
\inf_{v_{h}\in V_{h}}\Vert u-v_{h} \Vert_{1 }={\mathcal O}(h^{k}) \ \mbox{ and } \ \inf_{\xi_{h}\in Q_{h}}\Vert \lambda-\xi_{h} \Vert_{-1 }={\mathcal O}(h^{k}) ,
\end{equation}
when $u\in H^{k+1}(\Omega)$ and $\lambda\in H^{k-1}(\Omega)$.

We will use the following discrete negative norm 
in proving the stability condition:
\begin{equation}
\Vert \xi_{h}\Vert_{-1,h}^{2}=\sum_{K\in \Ch}h_{K}^{2}\Vert \xi_{h}\Vert_{0,K}^{2}  \quad \forall \xi_{h}\in Q_{h}.
\end{equation}
Analogously to the Stokes problem~\cite{S90}, we will need the following auxiliary result. Note that the result holds independently of the choice of  the finite element spaces. 
 \begin{lemma}\label{lem:aux} There exist positive constants $C_1$ and $C_2$ such that 
 \begin{equation}\label{aux}
 \sup_{v_h\in V_h} \frac{\langle  v_h, \xi_h\rangle   }{\Vert   v_h\Vert_{1}} \ge C_1 \Vert \xi_h\Vert_{-1} -C_2 \Vert \xi_{h}\Vert_{-1,h} \quad \forall \xi_h \in Q_h.
 \end{equation}
 \end{lemma}
 \begin{proof}
     The continous stability (Theorem~\ref{contstab}) implies that there exists $w \in H^1_0(\Omega)$ and $C>0$ such that
\begin{equation}
    \label{lem11:cinfsup}
    \langle w, \xi_h \rangle \geq C \|w\|_1\|\xi_h\|_{-1}
\end{equation}
for all $\xi_h \in Q_h$. Let $\widetilde{w} \in V_h$ be the Cl\'{e}ment interpolant \cite{C75} of $w$. Since $\xi_{h}\in L^{2}(\Omega)$ the duality pairing equals to the $L^{2}$-inner product. Then \eqref{lem11:cinfsup} and the Cauchy--Schwarz inequality give 
\begin{equation}
    \label{lem11:part1}
    \begin{aligned}
            \langle \widetilde{w}, \xi_h \rangle &= \langle \widetilde{w}-w,\xi_h \rangle + \langle w,\xi_h \rangle \\
                                                  &=  \sum_{K \in \mathcal{C}_h} (w-\widetilde{w},  \xi_h)_{K}+C \|w\|_1\|\xi_h\|_{-1}\\
                                                 &\geq - \sum_{K \in \mathcal{C}_h} \|w-\widetilde{w}\|_{0,K} \|\xi_h\|_{0,K}+C \|w\|_1\|\xi_h\|_{-1} \\
                                                 &= - \sum_{K \in \mathcal{C}_h} h_K^{-1} \|w-\widetilde{w}\|_{0,K} h_K \|\xi_h\|_{0,K}+C \|w\|_1 \|\xi_h\|_{-1} \\
                                                 &\geq - \Big(\sum_{K \in \mathcal{C}_h} h_K^{-2} \|w-\widetilde{w}\|_{0,K}^2\Big)^{\frac12} \|\xi_h\|_{-1,h} + C \|w\|_1 \|\xi_h\|_{-1}.
    \end{aligned}
\end{equation}
From the properties of the Cl\'{e}ment interpolant, we have
\[
    \Big(\sum_{K \in \Ch} h_K^{-2} \|w-\widetilde{w}\|^2_{0,K}\Big)^{\frac12} \leq C^\prime |w|_{1,K} \quad \text{and} \quad \|\widetilde{w}\|_1 \leq C^{\prime\prime} \|w\|_1
\]
which together with \eqref{lem11:part1} shows that
\begin{align*}
    \langle \widetilde{w}, \xi_h \rangle &\geq -C^\prime |w|_1 \|\xi_h\|_{-1,h}+C \|w\|_1 \|\xi_h\|_{-1} \\
                                         &\geq -C^\prime \|w\|_1 \|\xi_h\|_{-1,h}+C\|w\|_1\|\xi_h\|_{-1} \\
                                         &\geq C^{\prime\prime}(C\|\xi_h\|_{-1}-C^\prime\|\xi_h\|_{-1,h})\|\widetilde{w}\|_1.
\end{align*}
Dividing by $\|\widetilde{w}\|_1$ provides the claim.
\end{proof}
 
Using this  result one proves the following.
\begin{lemma}\label{lem:discstab} If we have stability in the discrete norm, i.e.
\begin{equation}\label{hstab}
\sup_{w_{h}\in V_{h}}\frac{\langle w_{h}, \xi_{h}\rangle   }{\Vert w_{h}\Vert_{1}}\gtrsim\Vert \xi_{h}\Vert _{-1,h} \quad \forall \xi_{h}\in Q_{h} \end{equation}
then the Babu$\check{\rm s}$ka--Brezzi condition \eqref{femstab} holds.
\end{lemma}
\begin{proof}
Suppose \eqref{hstab} holds.
Then by Lemma~\ref{lem:aux}, for $t>0$ we have
\begin{equation}
    \label{lem12:part1}
    \begin{aligned}
            \sup_{w_h \in V_h} \frac{\langle w_h, \xi_h\rangle}{\|w_h\|_1} &= t\sup_{w_h \in V_h} \frac{\langle w_h, \xi_h\rangle}{\|w_h\|_1}+(1-t)\sup_{w_h \in V_h} \frac{\langle w_h, \xi_h\rangle}{\|w_h\|_1} \\
                                                                           &\geq t(C_1 \|\xi_h\|_{-1}-C_2\|\xi_h\|_{-1,h})+(1-t)C_3\|\xi_h\|_{-1,h}\\
                                                                           &= t C_1 \|\xi_h\|_{-1} + (C_3-C_3t-C_2t)\|\xi_h\|_{-1,h}.
    \end{aligned}
\end{equation}
Thus, if we choose ~$t=\frac12 C_3(C_2+C_3)^{-1}$,
the second term on the right hand side of \eqref{lem12:part1} is positive and hence
\[
    \sup_{w_h \in V_h} \frac{\langle w_h, \xi_h\rangle}{\|w_h\|_1} \geq \frac{C_1 C_3}{2(C_2+C_3)} \|\xi_h\|_{-1} \quad \forall \xi_h \in Q_h.
\]
\end{proof}

The advantage in using the intermediate step in proving the stability in the mesh-dependent norm is that the discrete negative norm can be computed elementwise in contrast to the continuous norm which is global.
\begin{lemma} The finite element spaces \eqref{Vh}  and \eqref{Lh} satisfy the Babu$\check{\rm s}$ka--Brezzi condition \eqref{femstab}.
\end{lemma} 
\begin{proof}
We begin by showing stability in the discrete norm $\|\cdot\|_{-1,h}$ and then apply Lemma~\ref{lem:discstab} to get the stability in the continuous norm. Given $\xi_h \in Q_h$, we can define $w_h \in V_h$ by
\begin{equation}
    \label{eq:stabconst}
    w_h|_K=b_K h_K^2 \xi_h|_K.
\end{equation}
Then we estimate 
\begin{equation}
    \label{lem13:part1}
    \begin{aligned}
        \langle w_h, \xi_h \rangle &= \sum_{K \in \mathcal{C}_h} ( w_h, \xi_h )_K = \sum_{K \in \mathcal{C}_h} \int_K b_K h_K^2 \xi_h^2 \,\mathrm{d}x \gtrsim \| \xi_h \|_{-1,h}^2.
    \end{aligned}
\end{equation}
Moreover, using the inverse inequality and the definition \eqref{eq:stabconst} we get
\begin{equation}
    \label{lem13:part2}
    \begin{aligned}
            \|w_h\|_1^2 & \lesssim \sum_{K \in \mathcal{C}_h} h_K^{-2} \|w_h\|_{0,K}^2 \lesssim \sum_{K \in \mathcal{C}_h} h_K^2 \|\xi_h\|_{0,K}^2= \|\xi_h\|_{-1,h}^2.
    \end{aligned}
\end{equation}
Combining estimates \eqref{lem13:part1} and \eqref{lem13:part2} proves stability in the discrete norm. Finally, we apply Lemma~\ref{lem:discstab} to conclude the result.
\end{proof}

\begin{lemma} \label{lem:m1inverse}The following inverse estimate holds
$$ \|\xi_h\|_{-1,h}\lesssim \|\xi_h\|_{-1} \quad \forall \xi_{h}\in Q_{h}.$$
\end{lemma}
\begin{proof} In the preceeding lemma, we showed that
$$
\sup_{w_{h}\in V_{h}}\frac{\langle w_{h}, \xi_{h}\rangle   }{\Vert w_{h}\Vert_{1}}\gtrsim\Vert \xi_{h}\Vert _{-1,h} \quad \forall \xi_{h}\in Q_{h}. $$
The assertion thus follows from the fact that
$$ \vert \langle w_{h}, \xi_{h}\rangle  \vert  \lesssim 
\Vert w_{h}\Vert_{1}\Vert \xi_{h}\Vert _{-1}.
$$
\end{proof}
\begin{remark} {\rm Note that the above inverse inequality is valid in an arbitrary piecewise polynomial finite element space $\Lambda_{h}$, since one can always use the bubble function technique to construct a space $V_{h}$ in which the discrete stability inequality is valid. }
\end{remark}

 The a priori error estimate now follows from the discrete stability estimate of Theorem \ref{dsmixed}.
\begin{theorem}
The following error estimate holds
\[
\|u-u_h\|_1+\| \lambda -\lambda_h\|_{-1} \lesssim \inf_{v_h \in V_h} \| u-v_h\|_1 +
\inf_{\mu_h\in\Lambda_h} \big( \| \lambda-\mu_h\|_{-1} + \sqrt{ \langle u-g, \mu_h\rangle}\, \big)\,.
\]
\label{thm:mixedapriori}
\end{theorem}
\begin{proof}
\label{subsec:22}
 Let  $(v_h,\mu_h)\in V_h\times \Lambda_h$ be arbitrary. In view of  the stability estimate, there exists $w_h\in V_h$ such that
\begin{equation}
\|w_h\|_1 \lesssim \|u_h-v_h\|_1 +\|\lambda_h-\mu_h\|_{-1}
\label{wbound}
\end{equation}
and
\begin{equation}
\big(\|u_h-v_h\|_1+\|\lambda_h-\mu_h\|_{-1}\big)^2 \, \lesssim \mathcal{B}(u_h-v_h,\lambda_h-\mu_h;w_h,\mu_h-\lambda_h)
  .
 \label{bapriori1}
 \end{equation}
Given the discrete problem statement and the bilinearity of $\mathcal{B}$, by adding and subtracting $\mathcal{B}(u,\lambda;w_h,\mu_h-\lambda_h)$, we obtain
\begin{equation}
    \label{thm35:bound}
    \begin{aligned}
        &\mathcal{B}(u_h-v_h,\lambda_h-\mu_h;w_h,\mu_h-\lambda_h) \\
        &= \mathcal{B}(u_h,\lambda_h;w_h,\mu_h-\lambda_h) - \mathcal{B}(v_h,\mu_h;w_h,\mu_h-\lambda_h) \\
        &\leq
        \mathcal{B}(u-v_h,\lambda-\mu_h;w_h,\mu_h-\lambda_h)
        +\mathcal{L}(w_h,\mu_h-\lambda_h)
        -\mathcal{B}(u,\lambda;w_h,\mu_h-\lambda_h)\\
        &=\mathcal{B}(u-v_h,\lambda-\mu_h;w_h,\mu_h-\lambda_h) + \langle u-g,\mu_h-\lambda_h\rangle \\
        &\leq \mathcal{B}(u-v_h,\lambda-\mu_h;w_h,\mu_h-\lambda_h) + \langle u-g,\mu_h-\lambda \rangle\, ,
    \end{aligned}
\end{equation}
where in the last two lines we have used the continuous problem, i.e.
\[
    0 \leq \langle u-g, \lambda_h - \lambda \rangle.
\]

The continuity of the bilinear form $\mathcal{B}$ and the bound \eqref{wbound}  imply
\begin{equation}
    \label{thm35:bound2}
    \begin{aligned}
        &\mathcal{B}(u-v_h,\lambda-\mu_h;w_h,\mu_h-\lambda_h) \\
        \phantom{=}&\lesssim \big( \|u-v_h\|_1 + \|\lambda-\mu_h\|_{-1}\big) \,  \big( \|u_h-v_h\|_1+\|\lambda_h-\mu_h\|_{-1}\big).
    \end{aligned}
\end{equation}
Combining the estimates \eqref{bapriori1}, \eqref{thm35:bound} and \eqref{thm35:bound2} gives
\begin{align*}
    &\left(\|u_h-v_h\|_1 + \|\lambda_h-\mu_h\|_{-1}\right)^2 \\
    &\lesssim (\|u-v_h\|_1+\|\lambda-\mu_h\|_{-1})(\|u_h-v_h\|_1+\|\lambda_h-\mu_h\|_{-1})+\langle u-g, \mu_h-\lambda \rangle.
\end{align*}
Applying Young's inequality to the first term on the right hand side and completing the square gives
\[
\|u_h-v_h\|_1+\|\lambda_h-\mu_h\|_{-1} \lesssim \| u-v_h\|_1 +
\| \lambda-\mu_h\|_{-1} + \sqrt{\, \langle u-g,\mu_h-\lambda\rangle}.
\]
Since $\langle u-g, \lambda\rangle=0$, the triangle inequality yields
\begin{align*}
    \|u-u_h\|_1 + \|\lambda-\lambda_h\|_{-1} &\leq \|u-v_h\|_1 + \|u_h-v_h\|_1+\|\lambda-\mu_h\|_{-1}+\|\lambda_h-\mu_h\|_{-1} \\
                                             &\lesssim \inf_{v_h \in V_h} \|u-v_h\|_1 + \inf_{\mu_h \in \Lambda_h}\!\!\left(\|\lambda-\mu_h\|_{-1}+\sqrt{\langle u-g,\mu_h\rangle}\right)\!.
\end{align*}
\end{proof}
\begin{remark} This error estimate is known in the literature, cf. \cite[Theorem 5, Remark 3]{Has81} and  \cite[Lemma 4.3, Remark 4.9]{HHN}. However,  we perform  the analysis  more in the spirit of Babu\v{s}ka \cite{B70/71,B73} than that of Brezzi \cite{Br74} which seems to be the common approach.
\end{remark}

Next we derive the a posteriori estimate. We define the local error 
estimators $\eta_K$ and $\eta_E$ by
\begin{align*}
    \eta_K^2 &= h_K^2\, \|\Delta u_h+\lambda_h+f\|_{0,K}^2,\\ 
    \eta_E^2&= h_E\, \| \llbracket \nabla u_h \cdot {\bm n}\rrbracket  \|_{0,E}^2 .
\end{align*}
Further, we define
\begin{align*}
    \eta^2 &= \sum_{K\in\mathcal{C}_h} \eta_K^2 +\sum_{E\in\mathcal{E}_h} \eta_E^2,  \\
    S_m&= \|(g-u_h)_+\|_{1} + \sqrt{\langle (g-u_h)_+,\lambda_h\rangle},
\end{align*}
where $w_+=\max\{w,0\}$ denotes the positive part of $w$.
\begin{theorem} The following a posteriori estimate holds
\[
\|u-u_h\|_1+\|\lambda-\lambda_h\|_{-1}  \lesssim \eta +  S_m.
\]
\label{thm:mixedaposteriori}
\end{theorem}
\begin{proof}
    By the stability of the continuous problem (Theorem~\ref{contstab}), there exists $w\in V$ such that
\begin{equation}
\|w\|_1 \lesssim \|u-u_h\|_1+\|\lambda-\lambda_h\|_{-1}
\label{postnorm}
\end{equation}
and 
\begin{equation}
\big( \|u-u_h\|_1+\|\lambda-\lambda_h\|_{-1} \big)^{2} \lesssim \mathcal{B}(u-u_h,\lambda-\lambda_h;w,\lambda_h-\lambda)
 .
 \label{bapost1}
 \end{equation}
 Let $\widetilde{w}\in V_h$ be the Cl\'ement interpolant of $w$. The problem statement gives
\[
    0 \leq -\mathcal{B}(u_h,\lambda_h;-\widetilde{w},0)+\mathcal{L}(-\widetilde{w},0).
\]
It follows that
\begin{align*}
    &\big( \|u-u_h\|_1+\|\lambda-\lambda_h\|_{-1}\big) ^{2} \\
&\lesssim 
  \mathcal{B}(u,\lambda;w,\lambda_h-\lambda)-  \mathcal{B}(u_h,\lambda_h;w,\lambda_h-\lambda)\\ 
  &\lesssim
  \mathcal{B}(u,\lambda;w,\lambda_h-\lambda)-  \mathcal{B}(u_h,\lambda_h;w,\lambda_h-\lambda)-
     \, \mathcal{B}(u_h,\lambda_h;-\widetilde{w},0)+\mathcal{L}(-\widetilde{w},0) \\
  &\lesssim  \mathcal{L}(w,\lambda_h-\lambda)  -
    \mathcal{B}(u_h,\lambda_h;w-\widetilde{w},\lambda_h-\lambda)+\mathcal{L}(-\widetilde{w},0) 
    \\ 
  &\lesssim  \mathcal{L}(w-\widetilde{w},\lambda_h-\lambda)-\mathcal{B}(u_h,\lambda_h;w-\widetilde{w},\lambda_h-\lambda).
\end{align*}
Opening up the right hand side and combining terms, results in
\begin{align*}
    &\left(\|u-u_h\|_1+\|\lambda-\lambda_h\|_{-1}\right)^2 \\
    &\lesssim (f,w-\widetilde{w})-\langle g,\lambda_h-\lambda \rangle - (\nabla u_h, \nabla (w-\widetilde{w}))+\langle w-\widetilde{w} ,  \lambda_h \rangle + \langle u_h, \lambda_h -\lambda \rangle \\
    &= \sum_{K \in \Ch} (\Delta u_h+\lambda_h+f,w-\widetilde{w})_K-\sum_{E \in \Eh}(\llbracket \nabla u_h \cdot \boldsymbol{n} \rrbracket, w-\widetilde{w})_E+\langle u_h-g,\lambda_h-\lambda \rangle.
\end{align*}
The first two terms are estimated as usual; recall that the Cl\'ement interpolant satisfies
\[
\Big( \sum_{K\in\mathcal{C}_h} h_K^{-2} \| w-\widetilde{w}\|^2_{0,K} +\sum_{E\in\mathcal{E}_h} h_E^{-1}\|w-\widetilde{w}\|_{0,E}^2\Big)^{\frac12}\lesssim \|w\|_1 .
\] 
 The last term is bounded as follows:
\begin{align*}
    \langle u_h-g, \lambda_h-\lambda \rangle &= \langle g-u_h, \lambda \rangle \\
                                             &\leq \langle (g-u_h)_+,\lambda \rangle\\
                                             &= \langle (g-u_h)_+,\lambda-\lambda_h \rangle + \langle (g-u_h)_+, \lambda_h \rangle\\
                                             &\leq \|(g-u_h)_+\|_1 \|\lambda-\lambda_h\|_{-1} +\langle (g-u_h)_+, \lambda_h \rangle.
\end{align*}
\end{proof}
 
 To derive lower bounds,   
 let $f_h\in Q_h$ be the $L^2$-projection of $f$ and define 
\begin{align}
\osckf&=h_K\Vert f-f_h\Vert_{0,K}\ \mbox{ and }
\\
\oscf^2&=\sum_{K\in \Ch}\osckf^2.
\end{align}

A function $ w\in H^{1}_{0}(U)$, with $U\subset \Omega$, we extend by zero into $\Omega\setminus U$, and and for functions in $\mu\in \Lambda$ we then define
\begin{equation}
\Vert \mu \Vert _{-1,U}=\sup_{v\in H_{0}^{1}(U)}\frac{\langle  v,\mu\rangle}{\Vert v \Vert_{1,U} }.
\end{equation}

We   let $\omega(E)$ be the union of the two elements sharing $E$.

\begin{lemma}
\label{lowbound}
    For all $v_h\in V_h$ and $\mu_h\in Q _h$, it holds
\begin{align} 
\label{a0}
&h_{K}\Vert \Delta v_{h}+\mu_{h} + f  \Vert_{0,K}\lesssim 
  \Vert u-v_{h}\Vert_{1,K} + \Vert \lambda-\mu_{h}\Vert_{-1,K}+\osckf,
\\
\label{b0}
&\Big(\sum_{K\in \Ch}h_{K}^{2}\Vert \Delta v_{h}+\mu_{h} + f  \Vert_{0,K}^{2}\Big)^{\frac12}
 \lesssim \Vert  u-v_{h}\Vert_{ 1}+ \Vert \lambda -\mu_{h} \Vert_{-1}+\oscf,
\\ \label{d12} & 
h_E^{1/2}\, \| \llbracket \nabla v_h \cdot {\bm n}\rrbracket  \|_{0,E} \lesssim 
\Vert u-v_{h}\Vert_{1,\omega(E)} + \Vert \lambda-\mu_{h}\Vert_{-1,\omega(E)}+\sum_{K\subset \omega(E)}\osckf,
\\  
 & 
 \Big(\sum_{E \in \Eh }h_E \| \llbracket \nabla v_h \cdot {\bm n}\rrbracket  \|_{0,E}^{2}\Big)^{\frac12} \lesssim 
 \Vert u-v_{h}\Vert_{1 } + \Vert \lambda-\mu_{h}\Vert_{-1 }+ \oscf.
\label{c0}
\end{align}
\end{lemma}
\begin{proof} Using the bubble function $b_{K}\in P_{3}(K)\cap H^{1}_{0}(K)$, we define $\gamma_{K}$ by
$$\gamma_{K}=
h_K^2 b_{K}(\Delta v_{h}+\mu_{h} + f_{h}) \ \mbox{  in }K,\quad \text{and}\quad \gamma_{K}=0 ~~\mbox{~in~} \Omega \setminus K. $$
Testing with $\gamma_{K}$ in \eqref{dualvi}$_1$ yields
\begin{equation}
(\nabla u, \nabla \gamma_{K} )_{K} -\langle \gamma_{K}, \lambda \rangle =  (f,\gamma_{K})_{K}.
\end{equation}
Using this and the  norm equivalence in polynomial   spaces, we obtain
\begin{equation}\label{a1}
\begin{aligned}
    &h_K^2\Vert \Delta v_{h}+\mu_{h} + f_{h}\Vert_{0,K}^{2} \\
    &\lesssim h^2_K\Vert\sqrt{b_{K}}(\Delta v_{h}+\mu_{h} + f_{h})\Vert_{0,K}^{2}  \\  
  &=(\Delta v_{h}+\mu_{h} + f_{h}, \gamma_{K})_{K} \\
  &=(\Delta v_{h}+\mu_{h}, \gamma_{K})_{K} +
(f, \gamma_{K})_{K}+(  f_{h}-f, \gamma_{K})_{K} \\
&  =(\Delta v_{h}+\mu_{h}, \gamma_{K})_{K} +
(\nabla u, \nabla \gamma_{K} )_{K}-\langle \gamma_{K}, \lambda \rangle +(  f_{h}-f, \gamma_{K})_{K} \\
& = 
(\nabla (u-v_{h}), \nabla \gamma_{K} )_{K}+\langle \gamma_{K}, \mu_{h}-\lambda \rangle +(  f_{h}-f, \gamma_{K})_{K}.
\end{aligned}
\end{equation}
The right hand side above can be  estimated as follows
\begin{align*}\label{a2}
\nonumber
&(\nabla (u-v_{h}), \nabla \gamma_{K} )_{K}+\langle  \gamma_{K}, \mu_{h}-\lambda \rangle +(  f_{h}-f, \gamma_{K})_{K}
\\ 
& \leq \Vert \nabla (u-v_{h})\Vert_{0,K}   \Vert  \nabla \gamma_{K} \Vert _{0, K}
+\Vert \mu_{h}-\lambda \Vert_{-1,K} \Vert \gamma_{K}\Vert_{1,K}+\osckf h_{K}^{-1} \Vert \gamma_{K}\Vert_{0,K}.
 \nonumber
\end{align*}
By inverse inequalities, we have 
\begin{equation}\label{a3}
    \begin{aligned}
  \Vert \gamma_{K}\Vert_{1,K}^2& \lesssim  h_K^{-2} \Vert \gamma_{K} \Vert_{0,K}^2\\
                           &=  h^2_K\Vert  b_{K}(\Delta v_{h}+\mu_{h} + f_{h})\Vert_{0,K}^{2} \\
                           &\lesssim   h^2_K\Vert \Delta v_{h}+\mu_{h} + f_{h}\Vert_{0,K}^{2} .
    \end{aligned}
\end{equation}
Combining \eqref{a1}--\eqref{a3} gives the first estimate \eqref{a0}.

To prove \eqref{b0}, we write
 $\gamma=\sum_{K\in \Ch}\gamma_{K}$
and sum the inequality \eqref{a1} over all elements. This yields
\begin{equation}
    \begin{aligned}
\label{low1}
 &\sum_{K\in \Ch} h_K^2 \Vert\Delta v_{h}+\mu_{h} + f_{h}\Vert_{0,K}^{2}\\
 &\lesssim \sum_{K\in \Ch}\big\{ (\nabla (u-v_{h}) ,\nabla \gamma_{K} )_{K}+\langle  \gamma_{K}, \mu_{h}-\lambda \rangle +(  f_{h}-f, \gamma_{K})_{K}\big\} \\
 &=(\nabla (u-v_{h}) , \nabla \gamma)+\langle \gamma, \mu_{h}-\lambda \rangle +(  f_{h}-f, \gamma)\\
&\leq \Vert u-v_h \Vert_1\Vert \gamma\Vert_1 + \Vert  \mu_{h}-\lambda\Vert_{-1}\Vert \gamma\Vert_1 +\oscf \Big( \sum _{K\in \Ch} h_K^{-2} \Vert \gamma\Vert_{0,K}^2\Big)^{\frac12}.
\end{aligned}
\end{equation}
Summing estimates \eqref{a3}  over ${K\in \Ch}$,  leads to \eqref{b0}.

Next, let $b_{E}$ be the bubble on $E$, and denote $ \ve={b_{E}  \mathcal  X}_{E} ( \llbracket \nabla v_h \cdot {\bm n}\rrbracket)$, where ${\mathcal  X}_{E}$ is the standard extension operator onto  $ H_{0}^{1}(\omega(E))$ cf.~Braess~\cite{braess}. By scaling and Poincar\'e's inequality we have 
\begin{equation}\label{invc}
\begin{aligned}
  \| \llbracket \nabla v_h \cdot {\bm n}\rrbracket  \|_{0,E}&\approx \|\sqrt{b_{E}} \llbracket \nabla v_h \cdot {\bm n}\rrbracket  \|_{0,E} \approx h_{E}^{-1/2}\|  \ve  \|_{0,\omega(E)} 
   \\
   &\approx  h_{E}^{1/2}\|  \nabla  \ve  \|_{0,\omega(E)}\approx  h_{E}^{1/2}\|   \ve  \|_{1,\omega(E)}.
   \end{aligned}
\end{equation}
Using this, Green's formula, and the variational formulation \eqref{dualvi}, yield
\begin{equation}\label{f0}
\begin{aligned}
    &\| \llbracket \nabla v_h \cdot {\bm n}\rrbracket  \|_{0,E} ^{2} \\
    &\lesssim
 \|\sqrt{b_{E}} \llbracket \nabla v_h \cdot {\bm n}\rrbracket  \|_{0,E} ^{2}
= (  \llbracket \nabla v_h \cdot {\bm n}\rrbracket,    \ve)_{E}
\\&= (\Delta v_{h}, \ve)_{\omega(E)}+ (\nabla v_{h},\nabla  \ve)_{\omega(E)}
\\& = (\Delta v_{h}, \ve)_{\omega(E)}+ (\nabla (v_{h}-u),\nabla  \ve)_{\omega(E)}+
(\nabla u,\nabla  \ve)_{\omega(E)}
\\&= (\Delta v_{h}, \ve)_{\omega(E)}+ (\nabla (v_{h}-u),\nabla  \ve)_{\omega(E)}+
(f , \ve)_{\omega(E)}+\langle \lambda,  \ve\rangle
\\&= (\Delta v_{h}+\mu_{h}+f, \ve)_{\omega(E)}+ (\nabla (v_{h}-u),\nabla  \ve)_{\omega(E)}+
 \langle \lambda-\mu_{h},  \ve\rangle.
\end{aligned}
\end{equation}
Hence, it holds
\begin{equation}
\begin{aligned}  
\| \llbracket \nabla v_h \cdot {\bm n}\rrbracket  \|_{0,E} ^{2}&\lesssim   \| \Delta v_{h}+\mu_{h}+f \| _{0,\omega(E)}\|  \ve\|_{0,\omega(E)} \\
    &\phantom{=} + \|\nabla (u-v_{h})\| _{0,\omega(E)}\| \nabla  \ve\|_{0,\omega(E)}
    \\ &\phantom{=}   
+\| \lambda-\mu_{h} \|_{-1,\omega(E)} \|   \ve\| _{1,\omega(E)},
\end{aligned}
\end{equation}
and \eqref{invc} gives
\begin{equation}
    \begin{aligned}
        h_{E}^{1/2}\| \llbracket \nabla v_h \cdot {\bm n}\rrbracket  \|_{0,E}&\lesssim   h_{E}\| \Delta v_{h}+\mu_{h}+f \| _{0,\omega(E)}+ \|\nabla (u-v_{h})\| _{0,\omega(E)}\\
        &\phantom{=}+\| \lambda-\mu_{h} \|_{-1,\omega(E)} ,
    \end{aligned}
\end{equation}
and \eqref{d12}  follows from the already proved estimate \eqref{a0}.

In order to prove \eqref{c0}, we use \eqref{f0} to obtain
\begin{equation}\label{first}
\begin{aligned}
    \sum_{E\in \Eh}h_{E}\| \llbracket \nabla v_h \cdot {\bm n}\rrbracket  \|_{0,E} ^{2}&\lesssim 
\sum_{E\in \Eh}h_{E} (\Delta v_{h}+\mu_{h}+f, \ve)_{\omega(E)}
    \\&\phantom{=}+ \sum_{E\in \Eh}h_{E}(\nabla (v_{h}-u),\nabla  \ve)_{\omega(E)}\\
    &\phantom{=}+\sum_{E\in \Eh}h_{E}
 \langle \lambda-\mu_{h},  \ve\rangle.
\end{aligned}
\end{equation}
Since each $E\in \Eh$ contains two elements of $\Ch$,  we get from \eqref{invc}
 \begin{equation}
 \begin{aligned}
 \sum_{E\in \Eh}&h_{E}   (\Delta v_{h}+\mu_{h}+f, \ve)_{\omega(E)}
 \\
 &
 \leq 
  \Big(\sum_{E\in \Eh}h_{E}^{2} 
  \| \Delta v_{h}+\mu_{h}+f\|^{2}_{0,\omega(E)}\Big)^{\frac12}
\Big(\sum_{E\in \Eh}   h_{E}  ^{2} \|\ve\| _{0,\omega(E)} ^{2} \Big)^{\frac12}
  \\
  &
   \lesssim  \Big(\sum_{K\in \Ch}h_{K}^{2} \| \Delta v_{h}+\mu_{h}+f\| ^{2}_{0,K} \Big)^{\frac12}
\Big(\sum_{E\in \Eh}
h_{E}\| \llbracket \nabla v_h \cdot {\bm n}\rrbracket  \|_{0,E} ^{2} \Big)^{\frac12},
 \end{aligned}
 \end{equation}
 
  \begin{equation}
  \begin{aligned}
  \sum_{E\in \Eh}h_{E}&(\nabla (u-v_{h}),\nabla  \ve)_{\omega(E)},
 \\&   \leq 
  \Big(\sum_{E\in \Eh} 
  \| \nabla(u- v_{h})\|^{2}_{0,\omega(E)}\Big)^{\frac12}
\Big(\sum_{E\in \Eh} \| \nabla \ve\| _{\omega(E)} ^{2} \Big)^{\frac12}
  \\
  &
   \lesssim   \| \nabla(u- v_{h})\|_0
\Big(\sum_{E\in \Eh}
h_{E}\| \llbracket \nabla v_h \cdot {\bm n}\rrbracket  \|_{0,E} ^{2} \Big)^{\frac12},
  \end{aligned}
 \end{equation}
 and
  \begin{equation}\label{last}
  \begin{aligned}
  \sum_{E\in \Eh}h_{E}&
 \langle \ve, \lambda-\mu_{h} \rangle = \Big\langle  \sum_{E\in \Eh}h_{E}\ve,  \lambda-\mu_{h} \Big\rangle
\\
& \lesssim \|\lambda-\mu_{h}\| _{-1}   \Big\| \sum_{E\in \Eh}h_{E} \ve\Big\|_{1}
 \\&
 \lesssim \|\lambda-\mu_{h}\| _{-1}  \Big(\sum_{E\in \Eh}
h_{E}\| \llbracket \nabla v_h \cdot {\bm n}\rrbracket  \|_{0,E} ^{2} \Big)^{\frac12}
\end{aligned}
 \end{equation}
The asserted estimate now follows by combining \eqref{first}--\eqref{last} and \eqref{b0}.
\end{proof}

Choosing $v_{h}=u_{h}$ and $\mu_{h}=\lambda_{h}$ above, we obtain the local lower bounds
\begin{eqnarray}\eta_{K}&\lesssim& 
 \Vert u-u_{h}\Vert_{1,K} + \Vert \lambda-\lambda_{h}\Vert_{-1,K}+\osckf  ,
\\
  \eta_{E}&\lesssim& \Vert u-u_{h}\Vert_{1,\omega(E)} +\sum_{K\subset \omega(E)}\big(\Vert \lambda-\lambda_{h}\Vert_{-1,K}+ \osckf \big),
\end{eqnarray}
and the global bound
\begin{equation}
\eta \lesssim \|u-u_h\|_1+\| \lambda -\lambda_h\|_{-1} +\oscf.
\end{equation}
\begin{remark} These estimates  have appeared in the literature, cf. \cite{BS15,B05,BUSCH15}. For completeness, we have given a proof based on our approach.
\end{remark}
\begin{remark}
In proving Lemma \ref{lowbound}, we never used the fact that $(u_{h},\lambda_{h})$  solves the mixed problem. The estimates thus hold also  for the stabilized methods that will be presented in the next section. In fact, they turn out to be crucial for the a priori error analysis of these methods.
\end{remark}


\section{Stabilized methods}  From the Stokes problem, it is known that the technique of using stabilizing bubble degrees of freedom can be avoided by the so-called residual-based stabilizing \cite{HFB,SIAM}. Below we will show that this approach  applies also to the present problem. The resulting formulation, stability and error estimates are valid for any finite element pair $(V_{h}, \Lambda_{h})$.

Let us start by introducing the bilinear and linear forms $\mathcal{S}_h$ and $\mathcal{F}_h$ by
\begin{align*}
    \mathcal{S}_h(w,\xi;v,\mu) &= \sum_{K\in\mathcal{C}_h} h_K^2 (-\Delta w-\xi, -\Delta v-\mu)_K,
\\ 
\mathcal{F}_h(v,\mu)&=\sum_{K\in\mathcal{C}_h} h_K^2 (f, -\Delta v-\mu)_K,
\end{align*}
and then define the forms $\mathcal{B}_h$ and $\mathcal{L}_h$ through
\begingroup
\addtolength{\jot}{0.3em}
\begin{align*}
    \mathcal{B}_h(w,\xi;v,\mu)&= \mathcal{B}(w,\xi;v,\mu) - \alpha\, \mathcal{S}_h(w,\xi;v,\mu)  ,\\
    \mathcal{L}_h(v,\mu)&= \mathcal{L}(v,\mu) - \alpha \, \mathcal{F}_h(v,\mu)  ,
\end{align*}
\endgroup
where $\alpha>0$ is a stabilization parameter.

 Note that  with the assumption $f\in L^2(\Omega)$ it holds that  $\Delta u+  \lambda \in L^2(\Omega)$, even if 
  $\Delta u \not \in L^2(\Omega)$ and  $   \lambda \not\in L^2(\Omega)$. Hence it holds
 \begin{equation}
\mathcal{S}_h(u,\lambda;v_h,\mu_h) =  \mathcal{F}_h(v_h,\mu_h) \quad  \forall (v_h,\mu_h)\in V_h\times \Lambda_h .
\label{stermcons}
\end{equation}
This motivates the following
  stabilized finite element method.
  
  \textsc{The stabilized  method. }{\em
Find $(u_h,\lambda_h)\in V_h\times \Lambda_h$ such that
\begin{equation}
\mathcal{B}_h(u_h,\lambda_h;v_h,\mu_h-\lambda_h) \leq \mathcal{L}_h(v_h,\mu_h-\lambda_h) \quad \forall(v_h,\mu_h)\in V_h\times \Lambda_h.
\label{stabfem}
\end{equation}} 

In our  analysis, we need an inverse inequality which we write  as:   there exists a positive constant $C_{I}$ such that
\begin{equation}\label{invest}
C_{I}\sum_{K\in \Ch}h_{K}^{2}\Vert \Delta v_h \Vert_{0,K}^{2}\leq \Vert \nabla v_h\Vert_{0}^{2} \quad \forall v_h \in V_h.
\end{equation} 
The following stability condition holds.
\begin{theorem} \label{stabstab} Suppose that $ 0< \alpha < C_{I}$. It then holds:
for all $(v_{h},\xi_{h})\in V_{h}\times Q_{h}$, there exists $w_{h}\in V_{h}$, such that 
\begin{equation}
 \B_{h}(v_{h},\xi_{h};w_{h},-\xi_{h}) \gtrsim \big(   \Vert v _{h}\Vert_{1}+ \Vert \xi_{h} \Vert_{-1} \big)^{2}
\end{equation}
and 
\begin{equation}
  \|w _{h}\|_{1}\lesssim \Vert v _{h}\Vert_{1}+ \Vert \xi _{h}\Vert_{-1}.
\end{equation}
\end{theorem}
\begin{proof}
    Let $(v_h,\xi_h) \in V_h \times Q_h$  be  arbitrary. With the assumption $0<\alpha <C_{I}$, the inverse estimate \eqref{invest} gives
   \begin{align*}
        &\mathcal{B}_h(v_h,\xi_h;v_h,-\xi_h) 
         =\|\nabla v_h\|_0^2 -\alpha \sum_{K \in \Ch}h_{K}^{2} \|\Delta v_h\|_{0,K}^{2} + \alpha \|\xi_h\|_{-1,h}^2
       \\ & \geq \Big(1-\frac{\alpha}{C_{I}}\Big) \|\nabla v_h\|_0^2  + \alpha \|\xi_h\|_{-1,h}^2
       \\
       & \geq C_{3}\big(  \|\nabla v_h\|_0^2  +  \|\xi_h\|_{-1,h}^2 \big),
    \end{align*} 
    which guarantees stability with respect to the mesh-dependent norm for functions in $\Lambda_{h}$. To prove stability in the $H^{-1}$-norm, we let $q_h$ be the   function for which the supremum in Lemma~\ref{lem:aux} is obtained, viz.
     \begin{equation}
        \label{thmstabstab:pvtrick}
        \frac{\langle q_h,\xi_h\rangle}{\|q_h\|_1} \geq C_1 \|\xi_h\|_{-1}-C_2 \|\xi_h\|_{-1,h} .\end{equation}
       By homogeneity  we can assign the equality \begin{equation}
         \|q_h\|_1 = \|\xi_h\|_{-1}.
        \end{equation}
Using the above  relations,  estimate \eqref{invest}, and the Young's inequality, with $\varepsilon >0$,  gives 
\begin{align*}
\mathcal{B}_h&(v_h,\xi_h;q_h,0) =(\nabla v_{h},\nabla q_{h}) + \langle q_h,\xi_h\rangle
- \alpha   \sum_{K \in \Ch}h_{K}^{2} (\Delta v_h +\xi_{h}, \Delta q_{h})_{K}
\\ 
& 
\geq -\Vert \nabla v_{h} \Vert_{0} \Vert \nabla q_{h}\Vert_{0 }  + C_1 \|\xi_h\|_{-1}^{2}-C_2 \|\xi_h\|_{-1,h} \|\xi_h\|_{-1}
\\ 
& \quad - \alpha   \Big(\sum_{K \in \Ch}h_{K}^{2} \Vert \Delta v_h \Vert_{0,K}^{2}\Big)^{\frac12} 
 \Big(\sum_{K \in \Ch}h_{K}^{2} \Vert \Delta q_h \Vert_{0,K}^{2}\Big)^{\frac12} 
- \alpha \|\xi_h\|_{-1,h}  \Big(\sum_{K \in \Ch}h_{K}^{2} \Vert \Delta q_h \Vert_{0,K}^{2}\Big)^{\frac12} 
\\
& 
\geq -\Big(1+\frac{\alpha}{C_{I}}\Big)\Vert \nabla v_{h} \Vert_{0} \Vert \nabla q_{h}\Vert_{0 }  + C_1 \|\xi_h\|_{-1}^{2}-C_2 \|\xi_h\|_{-1,h} \|\xi_h\|_{-1}
 - \frac{\alpha}{\sqrt{C_I}} \|\xi_h\|_{-1,h} \|\nabla q_h \|_{0}
 \\
& 
\geq -\Big(1+\frac{\alpha}{C_{I}}\Big)\Vert \nabla v_{h} \Vert_{0} \Vert  \xi_{h}\Vert_{-1 }  + C_1 \|\xi_h\|_{-1}^{2}-C_2 \|\xi_h\|_{-1,h} \|\xi_h\|_{-1}
 - \frac{\alpha}{\sqrt{C_I}} \|\xi_h\|_{-1,h}  \|\xi_h \|_{-1}\\& 
\geq 
\Big(C_1-\frac{\varepsilon}{2}(1+\frac{\alpha}{C_{I}}+C_{2}+\alpha) \Big) \|\xi_h\|_{-1}^{2}
-\frac{1}{2\varepsilon }\Big( \Big(1+\frac{\alpha}{C_{I}}\Big)\Vert \nabla v_{h} \Vert_{0} ^{2} + (C_2+\frac{\alpha}{C_I}) \|\xi_h\|_{-1,h} ^{2}\Big)
\\
&\geq C _{4} \Vert \xi_{h}\Vert_{-1}^{2} - C_{5}\big(  \|\nabla v_h\|_0^2  +  \|\xi_h\|_{-1,h}^2 \big),
\end{align*}
where $\varepsilon$  has been chosen small enough. 
Hence 
\begin{equation}
\mathcal{B}_h(v_h,\xi_h ;v_{h}+\delta q_h,-\xi_{h})\geq 
 \delta C _{4} \Vert \xi_{h}\Vert_{-1}^{2} + (C_{3}- \delta C_{5})\big(  \|\nabla v_h\|_0^2  +  \|\xi_h\|_{-1,h}^2 \big)
\end{equation}
and the assertion follows by choosing $0 <\delta < C_{3}/C_{5}$.
\end{proof}

Next, we derive the a priori estimate. We  follow our analysis for the Stokes problem,  see \cite{SV15}, and  use a technique introduced by Gudi~\cite{Gudi}. 
The key ingredient is a  tool from the a posteriori error analysis, namely the estimate \eqref{b0} of Lemma \ref{lowbound}.

\begin{theorem}
The following a priori estimate holds
\[
\|u-u_h\|_1+\| \lambda -\lambda_h\|_{-1} \lesssim \inf_{v_h\in V_h} \| u-v_h\|_1 +
\inf_{\mu_h\in\Lambda_h} \big( \| \lambda-\mu_h\|_{-1} + \sqrt{\langle u-g,\mu_h\rangle}\, \big)+ \oscf
\]
\end{theorem}
\begin{proof}
     Theorem~\ref{stabstab} implies that there exists $w_h\in V_h$
    such that
    \[
        \|w_h\|_1 \lesssim \|u_h-v_h\|_1+\|\lambda_h-\mu_h\|_{-1}
    \]
    and
    \[
        \left(\|u_h-v_h\|_1+\|\lambda_h-\mu_h\|_{-1}\right)^2 \lesssim \mathcal{B}_h(u_h-v_h,\lambda_h-\mu_h;w_h,\mu_h-\lambda_h).
    \]
    We then estimate, similarly to the bound \eqref{thm35:bound} of Theorem~\ref{thm:mixedapriori}, 
    \begin{align*}
        &\mathcal{B}_h(u_h-v_h,\lambda_h-\mu_h;w_h,\mu_h-\lambda_h)\\
        &=\mathcal{B}_h(u_h,\lambda_h;w_h,\mu_h-\lambda_h)-\mathcal{B}_h(v_h,\mu_h;w_h,\mu_h-\lambda_h)\\
                &\leq \mathcal{B}(u-v_h,\lambda-\mu_h;w_h,\mu_h-\lambda_h)+\mathcal{L}(w_h,\mu_h-\lambda_h)-\mathcal{B}(u,\lambda;w_h,\mu_h-\lambda_h)\\
        &\phantom{=}-\alpha\sum_{K \in \Ch} h_K^2 (-\Delta(u-v_h)-(\lambda-\mu_h),-\Delta w_h-(\mu_h-\lambda_h))_K\\
        &\leq\mathcal{B}(u-v_h,\lambda-\mu_h;w_h,\mu_h-\lambda_h)+\langle u - g, \mu_h-\lambda_h \rangle\\
        &\phantom{=}+\alpha \sum_{K \in \Ch} h_K^2 \|\Delta v_h+\mu_h+f\|_{0,K} \|\Delta w_h + \mu_h - \lambda_h\|_{0,K}.
    \end{align*}
    The first two terms are as in Theorem~\ref{thm:mixedapriori}.
    The last term is estimated using the triangle and Cauchy--Schwarz inequalities, Lemma~\ref{lowbound}, and the inverse inequality of Lemma~\ref{lem:m1inverse}:
    \begin{align*}
        &\sum_{K \in \Ch} h_K^2\|\Delta v_h+\mu_h+f\|_{0,K} \|\Delta w_h+\mu_h-\lambda_h\|_{0,K}\\
        &\leq\Big(\sum_{K \in \Ch} h_K^2 \|\Delta v_h + \mu_h +f\|^2_{0,K}\Big)^{\frac12}\Big(\sum_{K\in\Ch}h_K^2\|\Delta w_h\|_{0,K}^2\Big)^{\frac12}\\
        &\phantom{=}+\Big(\sum_{K \in \Ch} h_K^2 \|\Delta v_h + \mu_h +f\|^2_{0,K}\Big)^{\frac12}\Big(\sum_{K\in\Ch}h_K^2\|\mu_h-\lambda_h\|_{0,K}^2\Big)^{\frac12}\\
        &\lesssim (\|u-v_h\|_1+\|\lambda-\mu_h\|_{-1}+\oscf)(\|\nabla w_h\|_0+\|\mu_h-\lambda_h\|_{-1,h})
        \\
        &\lesssim (\|u-v_h\|_1+\|\lambda-\mu_h\|_{-1}+\oscf)(\|\nabla w_h\|_0+\|\mu_h-\lambda_h\|_{-1}).
    \end{align*}
\end{proof}

For the a posteriori estimate we define
\begin{equation}
    S_s= \|(g-u_h)_+\|_{1} + \sqrt{\langle (u_h-g)_+,\lambda_h\rangle}.
\end{equation}
\begin{theorem}\label{apostthstab}
The following a posteriori estimate holds  
\[
\|u-u_h\|_1+\|\lambda-\lambda_h\|_{-1}  \lesssim \eta + S_s.
\]
\label{thm:stabaposteriori}
\end{theorem}
\begin{proof}
    By the continuous stability there exists $w \in V$ with the properties
    \begin{equation}
        \|w\|_1 \lesssim \|u-u_h\|_1 + \|\lambda-\lambda_h\|_{-1},
    \end{equation}
    and
    \begin{equation}
        \left(\|u-u_h\|_1+\|\lambda-\lambda_h\|_{-1}\right)^2 \lesssim \mathcal{B}(u-u_h,\lambda-\lambda_h;w,\lambda_h-\lambda).
    \end{equation}
    Using the problem statement we have
    \[
        0 \leq -\mathcal{B}_h(u_h,\lambda_h;-\widetilde{w},0)+\mathcal{L}_h(-\widetilde{w},0),
    \]
    where $\widetilde{w}$ is the Cl\'ement interpolant of $w$.
    It follows that
    \begin{equation}
        \label{eq:stabapostpf}
        \begin{aligned}
            &\left(\|u-u_h\|_1+\|\lambda-\lambda_h\|_{-1}\right)^2 \\
            &\lesssim \mathcal{B}(u,\lambda;w,\lambda_h-\lambda)-\mathcal{B}(u_h,\lambda_h;w,\lambda_h-\lambda)\\
            &\lesssim \mathcal{L}(w,\lambda_h-\lambda)-\mathcal{B}(u_h,\lambda_h;w,\lambda_h-\lambda)-\mathcal{B}_h(u_h,\lambda_h;-\widetilde{w},0)+\mathcal{L}_h(-\widetilde{w},0)\\
            &\lesssim \mathcal{L}(w-\widetilde{w},\lambda_h-\lambda) - \mathcal{B}(u_h,\lambda_h,w-\widetilde{w},\lambda_h-\lambda)\\
            &\phantom{=} + \alpha \sum_{K \in \Ch} h_K^2 (-\Delta u_h-\lambda_h-f,\Delta \widetilde{w})_K.
        \end{aligned}
    \end{equation}
    The first two terms are estimated similarly as in the proof of Theorem~\ref{thm:mixedaposteriori} with the exception of the term $\langle u_h-g,\lambda_h-\lambda \rangle$.
   For the stabilized method, $\langle u_h-g,\lambda_h \rangle \neq 0$, and therefore
    \begin{align*}
        \langle u_h-g,\lambda_h-\lambda \rangle &\leq \langle (g-u_h)_+,\lambda \rangle - \langle g-u_h,\lambda_h \rangle \\
                                                &= \langle (g-u_h)_+,\lambda-\lambda_h \rangle + \langle (u_h-g)_+,\lambda_h \rangle \\
                                                &\leq \|(g-u_h)_+\|_1 \|\lambda-\lambda_h\|_{-1} + \langle (u_h-g)_+, \lambda_h \rangle.
    \end{align*}
    The last term in \eqref{eq:stabapostpf} is bounded using the Cauchy--Schwarz inequality and inverse estimate as
    \[
        \sum_{K \in \mathcal{C}_h} h_K^2 (-\Delta u_h-\lambda_h-f,\Delta \widetilde{w})_K \lesssim \Big(\sum_{K \in \mathcal{C}_h} h_K^2 \|\Delta u_h+\lambda_h+f\|_0^2\Big)^{\frac12} \|\nabla \widetilde{w}\|_0
    \]
    and by the properties of the Cl\'{e}ment interpolant we have that $\|\nabla \widetilde{w}\|_0\lesssim\|w\|_1$.
\end{proof}

Note that we have not explicitly defined the finite element spaces, and hence the method is stable and the estimate holds for all choices of finite element pairs. The optimal choice is dictated by the approximation properties and is
\begin{equation} V_h = 
\{ v_h\in V  : v_h\vert _K \in P_k(K) \ \forall K \in \Ch \},  
\end{equation}
and 
\begin{equation}
Q_h = \begin{cases}
    \{ \xi_h \in Q : \xi_h\vert_K \in P_{0} (K) \ \forall K \in \Ch \} &\mbox{ for } k=1,\\
    \{ \xi_h \in Q : \xi_h\vert_K \in P_{k-2} (K) \ \forall K \in \Ch \} &\mbox{ for } k\geq 2.
\end{cases}
\end{equation}
\begin{remark} 
\label{elimination} For the lowest order mixed method, i.e. for the method \eqref{mixedm} with $k=1$ in \eqref{Vh} and \eqref{Lh}, a local elimination of the bubble degrees of freedom gives the stabilized formulation with $k=1$, for which we have
$$
    \mathcal{S}_h(w,\xi;v,\mu)  = \sum_{K\in\mathcal{C}_h} h_K^2 (  \xi,   \mu)_K
\  \mbox{ and } \   
\mathcal{F}_h(v,\mu) =-\sum_{K\in\mathcal{C}_h} h_K^2 (f,  \mu)_K.
$$  This is in complete analogy with the relationship between the MINI and the Brezzi--Pitk\"aranta methods for the Stokes equations, cf. \cite{BP,Pierre}. Note also that  no upper bound needs to be imposed on $\alpha$ in this case.

\end{remark}

\section{Iterative solution algorithms}

The contact area, i.e.~the subset of $\Omega$ where the solution satisfies $u=g$, is  unknown and must be solved as a part of the solution process. Let us first consider the mixed method \eqref{mixedm}. The weak form corresponding to problem \eqref{mixedm} reads: find $(u_h,\lambda_h)\in V_h \times \Lambda_h$ such that
\begin{alignat}{2}
    (\nabla u_h,\nabla v_h) - (\lambda_h, v_h) &= (f,v_h), \quad && \forall v_h \in V_h, \\
    (u_h-g, \mu_h-\lambda_h ) &\geq 0, \quad && \forall \mu_h \in \Lambda_h. \label{eq:mixedweaksecond}
\end{alignat}
Testing the inequality \eqref{eq:mixedweaksecond} with $\mu_h=0$ and $\mu_h=2\lambda_h$ leads to the system
\begin{alignat*}{2}
    (\nabla u_h,\nabla v_h) - ( \lambda_h, v_h ) &= (f,v_h), \quad && \forall v_h \in V_h, \\
    ( u_h-g, \mu_h ) &\geq 0, \quad && \forall \mu_h \in \Lambda_h, \\
    ( u_h-g, \lambda_h ) &=0. &&~
\end{alignat*}
We consider the case of low order elements with  $1\leq k \leq 3$, and let $\xi_j$, $j\in\{1,\dots\!\,,M\}$, be the Lagrange (nodal) basis for $Q_h$. When writing $\mu_h =\sum_{j=1}^M \mu_j \xi_j$, we then have the characterization
\begin{equation}
\Lambda_{h}= \Big\{   \mu_{h} =\sum_{j=1}^M \mu_{j} \xi_j \ \vert \, \mu_{j}\geq 0 \, \Big\}.
\end{equation}
By letting $\varphi_j$, $j\in\{1,\dots\!\,,N\}$, be the  basis for $V_h$, and writing $u_h = \sum_{j=1}^N u_j \varphi_j$,    we arrive at the finite dimensional complementarity problem
\begin{align}
  \boldsymbol{A}\boldsymbol{u}-\boldsymbol{B}^T \boldsymbol{\lambda} &= \boldsymbol{f}, \\
\boldsymbol{B} \boldsymbol{u} &\geq \boldsymbol{g}, \label{eq:ncp1} \\
  \boldsymbol{\lambda}^T(\boldsymbol{B}\boldsymbol{u}-\boldsymbol{g})&=0, \label{eq:ncp2} \\
  \boldsymbol{\lambda} &\geq \boldsymbol{0}, \label{eq:ncp3}
\end{align}
where
\begin{alignat*}{7}
    \boldsymbol{A} &\in \mathbb{R}^{N \times N} \quad &&(\boldsymbol{A})_{ij}&&=(\nabla \varphi_i, \nabla \varphi_j), \qquad &&\boldsymbol{B} &&\in \mathbb{R}^{M \times N} \quad &&(\boldsymbol{B})_{ij}&&=(\xi_i,\varphi_j), \\
    \boldsymbol{f} &\in \mathbb{R}^N \quad &&~(\boldsymbol{f})_{i}&&=(f,\varphi_i), \qquad &&~\boldsymbol{g} &&\in \mathbb{R}^M \quad &&~(\boldsymbol{g})_i&&=(g,\xi_i),\\
    \boldsymbol{u} &\in \mathbb{R}^N \quad &&~(\boldsymbol{u})_i&&= u_i,\qquad &&~\boldsymbol{\lambda} &&\in \mathbb{R}^M \quad &&~(\boldsymbol{\lambda})_i &&= \lambda_i.
\end{alignat*}

\begin{remark}
  For higher order methods with $k>3$ and a nodal basis the inequalities $\lambda_h \geq 0$ and $\boldsymbol{\lambda}\geq\boldsymbol{0}$ are not  equivalent and another solution strategy is required if one wants the solution space to span all positive piecewise polynomials. 
\end{remark}

Following e.g.~Ulbrich~\cite{U11} the three constraints \eqref{eq:ncp1}--\eqref{eq:ncp3} can be written as a single nonlinear equation to get
\begin{equation}
    \label{eq:ncpsystem}
    \begin{aligned}
    \boldsymbol{A}\boldsymbol{u}-\boldsymbol{B}^T \boldsymbol{\lambda} &= \boldsymbol{f}, \\
        \boldsymbol{\lambda}-\max\{\boldsymbol{0},\boldsymbol{\lambda}+c(\boldsymbol{g}-\boldsymbol{B}\boldsymbol{u})\}&=\boldsymbol{0},
    \end{aligned}
\end{equation}
with any $c>0$. Application of the semismooth Newton method to the system \eqref{eq:ncpsystem} leads to Algorithm~\ref{alg:pdmixed}~\cite{HIK03}.
In the algorithm definition we use a notation similar to Golub--Van Loan~\cite{GvL} where, given a matrix $\boldsymbol{C}$
and a row position vector $\boldsymbol{i}$, we denote by $\boldsymbol{C}(\boldsymbol{i},:)$ the submatrix formed by the
rows of $\boldsymbol{C}$ marked by the index vector $\boldsymbol{i}$. Similarly, $\boldsymbol{b}(\boldsymbol{i})$ consists of
the components of vector $\boldsymbol{b}$ whose indices appear in vector $\boldsymbol{i}$.
Note that the linear system to be solved at each iteration step (Step 8) has  the saddle point structure. For this class of problems 
there exists numerous efficient solution methods, cf.~Benzi et al.~\cite{BGL}.

\begin{algorithm}
    \caption{Primal-dual active set method for the mixed problem}
    \label{alg:pdmixed}
    \begin{algorithmic}[1]
        \State $k = 0$; $\boldsymbol{\lambda}^0 = \boldsymbol{0}$
        \State Solve $\boldsymbol{A} \boldsymbol{u}^0 = \boldsymbol{f}$
        \While{$k<1$ or $\|\boldsymbol{\lambda}^k-\boldsymbol{\lambda}^{k-1}\|>\it{TOL}$}
            \State $\boldsymbol{s}^k = \boldsymbol{\lambda}^k + c(\boldsymbol{g}-\boldsymbol{B}\boldsymbol{u}^k)$
            \State Let $\boldsymbol{i}^k$ consist of the indices of the nonpositive elements of $\boldsymbol{s}^k$
            \State Let $\boldsymbol{a}^k$ consist of the indices of the positive elements of $\boldsymbol{s}^k$
            \State $\boldsymbol{\lambda}^{k+1}(\boldsymbol{i}^k) = \boldsymbol{0}$
            \State Solve
            \[
                \begin{bmatrix}
                    \boldsymbol{A} & -\boldsymbol{B}(\boldsymbol{a}^k,:)^T \\
                    -\boldsymbol{B}(\boldsymbol{a}^k,:) & \boldsymbol{0}
                \end{bmatrix}
                \begin{bmatrix}
                    \boldsymbol{u}^{k+1} \\
                    \boldsymbol{\lambda}^{k+1}(\boldsymbol{a}^k)
                \end{bmatrix}
                =
                \begin{bmatrix}
                    \boldsymbol{f} \\
                    -\boldsymbol{g}(\boldsymbol{a}^k)
                \end{bmatrix}
            \]
            \State $k = k+1$
        \EndWhile
    \end{algorithmic}
\end{algorithm}

Let us next consider the stabilized method \eqref{stabfem}. The respective discrete weak formulation is: find $(u_h,\lambda_h) \in V_h \times \Lambda_h$ such that
\begin{align*}
    (\nabla u_h, \nabla v_h)-( \lambda_h, v_h)-\alpha\!\!\sum_{K \in \mathcal{C}_h}\!\!h_K^2(\Delta u_h\!+\!\lambda_h,\Delta v_h)_K &= (f,v_h)+\alpha \!\!\sum_{K \in \mathcal{C}_h}\!\!h_K^2(f,\Delta v_h)_K, \\
    ( u_h\!-\!g, \mu_h\!-\!\lambda_h ) +\alpha\!\!\sum_{K \in \mathcal{C}_h}\!\!h_k^2(\Delta u_h\!+\!\lambda_h\!+\!f,\mu_h\!-\!\lambda_h)_K &\geq 0,
\end{align*}
hold for every $(v_h, \mu_h) \in V_h \times \Lambda_h$.

Through similar steps as in the mixed case we arrive at the algebraic system
\begin{align}
  \boldsymbol{A}_\alpha\boldsymbol{u}-\boldsymbol{B}^T_\alpha \boldsymbol{\lambda} &= \boldsymbol{f}_\alpha, \\
    \boldsymbol{B}_\alpha \boldsymbol{u} + \boldsymbol{C}_\alpha \boldsymbol{\lambda} &\geq \boldsymbol{g}_\alpha, \label{eq:sncp1} \\
    \boldsymbol{\lambda}^T(\boldsymbol{B}_\alpha \boldsymbol{u}+\boldsymbol{C}_\alpha \boldsymbol{\lambda}-\boldsymbol{g}_{\alpha})&=0, \label{eq:sncp2} \\
  \boldsymbol{\lambda} &\geq \boldsymbol{0}, \label{eq:sncp3}
\end{align}
where
\begin{alignat*}{3}
    \boldsymbol{A}_\alpha &\in \mathbb{R}^{N \times N} \quad &&(\boldsymbol{A}_\alpha)_{ij}&&=(\nabla \varphi_i, \nabla \varphi_j)-\alpha\!\!\sum_{K \in \mathcal{C}_h} h_K^2(\Delta \varphi_i, \Delta \varphi_j)_K,\\
    \boldsymbol{B}_\alpha &\in \mathbb{R}^{M \times N} \quad &&(\boldsymbol{B}_\alpha)_{ij}&&=(\xi_i,\varphi_j)+\alpha\!\!\sum_{K \in \mathcal{C}_h} h_K^2(\xi_i,\Delta \varphi_j)_K, \\
    \boldsymbol{C}_\alpha &\in \mathbb{R}^{M \times M} \quad &&(\boldsymbol{C}_\alpha)_{ij}&&=\alpha\!\!\sum_{K \in \mathcal{C}_h} h_K^2 (\xi_i,\xi_j)_K, 
    \\
    \boldsymbol{f}_\alpha &\in \mathbb{R}^N \quad &&~(\boldsymbol{f}_\alpha)_{i}&&=(f,\varphi_i)+\alpha\!\!\sum_{K \in \mathcal{C}_h}h_K^2(f,\Delta \varphi_i)_K, 
    \\
    \boldsymbol{g}_\alpha &\in \mathbb{R}^M \quad &&~(\boldsymbol{g}_\alpha)_i&&=(g,\xi_i)-\alpha\!\!\sum_{K \in \mathcal{C}_h}h_K^2(f,\xi_i)_K.
\end{alignat*}

The system corresponding to \eqref{eq:ncpsystem} reads
\begin{equation}
    \label{eq:ncpsystemstab}
    \begin{aligned}
        \boldsymbol{A}_\alpha\boldsymbol{u}-\boldsymbol{B}_\alpha^T \boldsymbol{\lambda} &= \boldsymbol{f}_\alpha, \\
        \boldsymbol{\lambda}-\max\{\boldsymbol{0},\boldsymbol{\lambda}+c(\boldsymbol{g}_\alpha-\boldsymbol{B}_\alpha\boldsymbol{u}-\boldsymbol{C}_\alpha \boldsymbol{\lambda})\}&=\boldsymbol{0},
    \end{aligned}
\end{equation}
which leads to Algorithm~\ref{alg:pdstab}. Note that the inversion of  the matrix $\boldsymbol{C}_\alpha$ is performed on each element separately, and that equation to be solved in Step 6 is symmetric and positive-definite. It has a condition number of $\mathcal{O}(h^{-2})$ and hence standard iterative solvers can be used. 
\begin{algorithm}
    \caption{Primal-dual active set method for the stabilized problem}
    \label{alg:pdstab}

    \begin{algorithmic}[1]
        \State $k = 0$
        \State Solve $\boldsymbol{A}_\alpha \boldsymbol{u}^0 = \boldsymbol{f}_\alpha$
        \State $\boldsymbol{\lambda}^0 = \boldsymbol{C}_\alpha^{-1}(\boldsymbol{g}_\alpha-\boldsymbol{B}_\alpha \boldsymbol{u}^0)$
        \While{$k<1$ or $\|\boldsymbol{\lambda}^k-\boldsymbol{\lambda}^{k-1}\|>\it{TOL}$}
            \State Let $\boldsymbol{a}^k$ consist of the indices of the positive elements of $\boldsymbol{\lambda}^k$
            \State Solve
            \begin{align*}
                &\left(\boldsymbol{A}_\alpha+\boldsymbol{B}_\alpha(\boldsymbol{a}^k,:)^T \boldsymbol{C}_\alpha(\boldsymbol{a}^k,\boldsymbol{a}^k)^{-1} \boldsymbol{B}_\alpha(\boldsymbol{a}^k,:)\right)\boldsymbol{u}^{k+1} \\
                &\qquad= \boldsymbol{f}_\alpha+\boldsymbol{B}_\alpha(\boldsymbol{a}^k,:)^T\boldsymbol{C}_\alpha(\boldsymbol{a}^k,\boldsymbol{a}^k)^{-1} \boldsymbol{g}_\alpha(\boldsymbol{a}_k)
            \end{align*}
            \State $\boldsymbol{\lambda}^{k+1} = \max\{\boldsymbol{0},\boldsymbol{C}_\alpha^{-1}(\boldsymbol{g}_\alpha-\boldsymbol{B}_\alpha \boldsymbol{u}^{k+1})\}$
            \State $k = k+1$
        \EndWhile
    \end{algorithmic}
\end{algorithm}
\section{Nitsche and penalty methods}
 Consider the stabilized method and recall that the stability and error estimates  hold for any finite element subspace $\Lambda_{h}$ for the reaction force.   Let  $\Ohc$  denote the contact region and assume, for the time being, that  its boundary   lies on the inter-element edges. We will derive the Nitsche's formulation by the following line of argument.

Noting that the functions of $\Lambda_h$ are discontinuous, we may eliminate the variable $\lambda_h$ locally on each element. Testing with $(0,\mu_h)$ in the stabilized problem \eqref{stabfem} gives
\[
    (u_h, \mu_h-\lambda_h) + \alpha\!\!\sum_{K\in\mathcal{C}_h} h_K^2 (\Delta u_h+\lambda_h,\mu_h-\lambda_h)_K \geq (g,\mu_h-\lambda_h)-\alpha\!\!\sum_{K \in \mathcal{C}_h} h_K^2 (f,\mu_h-\lambda_h)_K
\]
for every $\mu_h \in \Lambda_h$.
Since the contact area $\Ohc$ is assumed to be known, this reads 
\[
    (u_h,\mu_h)+\alpha\!\!\sum_{K \subset \Ohc} h_K^2(\Delta u_h+\lambda_h,\mu_h)_K=(g,\mu_h)-\alpha\!\!\sum_{K \subset \Ohc}h_K^2(f,\mu_h)_K
\]
giving locally
\begin{equation}
    \label{eq:locallambda}
    \lambda_h|_K = (\alpha h_K^2)^{-1} (\PK g - \PK u_h)|_K - \PK f|_K - \PK \Delta u_h|_K\quad \forall K \subset \Ohc,
\end{equation}
where $\PK$ is the $L^2$-projection onto $\Lambda_h$.  Testing with $(v_h,0)$ in \eqref{stabfem}
and substituting \eqref{eq:locallambda} into the resulting equation gives the problem: find $u_h \in V_h$ such that
\begin{align*}
    &(\nabla u_h, \nabla v_h)+\sum_{K \subset \Ohc}(\PK u_h,\PK \Delta v_h)_K + \sum_{K \subset \Ohc}(\PK \Delta u_h, \PK v_h)_K \\
    &+ \alpha^{-1}\!\!\sum_{K \subset \Ohc} h_K^{-2} (\PK u_h,\PK v_h)_K+\alpha\!\!\sum_{K \subset \Ohc} h_K^2 ((I-\PK) \Delta u_h, (I-\PK) \Delta v_h)_K \\
    &-\alpha\!\!\sum_{K \subset \Omega \setminus \Ohc} h_K^2 (\Delta u_h, \Delta v_h)_K \\
    &=(f,v_h)_{\Omega\setminus \Oc} + ((I-\PK)f,v_h)_{\Oc}+ \alpha\!\!\sum_{K \subset \Ohc} h_K^2 ((I-\PK)f,\Delta v_h)_K \\
    &\phantom{=}+\sum_{K \subset \Ohc}(\PK g,\PK \Delta v_h)_K + \alpha^{-1}\!\!\sum_{K \subset \Ohc} h_K^{-2}(\PK g,\PK v_h)_K \\
    &\phantom{=}+ \alpha\!\!\sum_{K \subset \Omega \setminus \Ohc} h_K^2 (f,\Delta v_h)_K,
\end{align*}
for every $v_h \in V_h$.

Now we are free to choose 
$$Q_h =\{ \xi_h\in Q:  \xi_h\vert _K\in P_{k}(K) \ \forall K\in \Ch\}.$$
 Then the formulation simplifies    to: find $u_h \in V_h$ such that
\begin{align*}
    &(\nabla u_h, \nabla v_h)+\sum_{K \subset \Ohc}(u_h,\Delta v_h)_{K} + \sum_{K \subset \Ohc}(\Delta u_h,v_h)_{K} + \alpha^{-1}\!\!\sum_{K \subset \Ohc} h_K^{-2} ( u_h, v_h)_K \\
    &-\alpha\!\!\sum_{K \subset \Omega \setminus \Ohc} h_K^2 (\Delta u_h, \Delta v_h)_K \\
    &=(f,v_h)_{\Omega\setminus \Oc} + \sum_{K \subset \Ohc}(g,\Delta v_h)_{K} + \alpha^{-1}\!\!\sum_{K \subset \Ohc} h_K^{-2}(g,v_h)_K + \alpha\!\!\sum_{K \subset \Omega \setminus \Ohc} h_K^2 (f,\Delta v_h)_K,
\end{align*}
holds for every $v_h \in V_h$. Note that the only thing that now remains of the discrete Lagrange multiplier is the rule to determine the contact region, i.e.~the elements $K$ for which  formula \eqref{eq:locallambda} yields a positive value for $\lambda_h$. 

This motivates the  formulation of Nitsche's method in the general case where the contact region is arbitrary. Given $ v_{h}\in V_{h} $ and the local mesh lengths $h_{K}$, 
we define the  $L^{2}(\Omega)$ functions $ \h$  and $\Delta_{h}v_{h} $ by 
\begin{equation} \h \vert_{K}=h_{K} ,\  \mbox{ and } \ 
\Delta_{h}v_{h}\vert _{K}=\Delta v_{h}\vert_{K }, \ K \in \Ch,\end{equation}
respectively. The discrete contact force is then defined as 
\begin{equation}\label{react}
 \lambda_{h}(x,y) = \max\{  0    , \big( (\alpha \h^2)^{-1} (  g -   u_h) -  f    -  \Delta_{h} u_h \big)(x,y) \,\}
\end{equation}
and the contact region is 
\begin{equation}
\Ohc = \{ \,(x,y)\in \Omega \, \vert \, \lambda_{h}(x,y)>0\, \}.
\end{equation}
 
\textsc{The Nitsche's method. }{\em 
Find $u_h \in V_h$ and $\Ohc=\Ohc(u_{h})$, such that
\begin{align*}
    &(\nabla u_h, \nabla v_h)+ (u_h,\Delta_h v_h)_{\Ohc} + (\Delta_h u_h,v_h)_{\Ohc} + \alpha^{-1}( \h^{-2}u_h, v_h)_{\Ohc} \\
    & - \alpha (\h^2 \Delta_h u_h, \Delta_h v_h)_{\Omega \setminus \Ohc}\\
    &=(f,v_h)_{\Omega\setminus \Ohc} +  (g,\Delta_h v_h)_{\Ohc} + \alpha^{-1} (\h^{-2}g,v_h)_{\Ohc} + \alpha (\h^2 f, \Delta_h v_h)_{\Omega \setminus \Ohc},
\end{align*}
holds for every $v_h \in V_h$.
} 

The iteration in Algorithm 2 corresponds now to solving the problem by updating the contact force through
\begin{equation}\label{itlambda}
 \lambda_{h}^{k+1}(x,y)= \max\{  0    , \big( (\alpha \h^2)^{-1} (  g -   u_h^{k}) -  f    -  \Delta_{h} u_h^{k} \big)(x,y) \,\}
\end{equation}
and  computing the contact area from
\begin{equation}
\Ohck = \{ \,(x,y)\in \Omega \, \vert \, \lambda_{h}^{k+1}(x,y)>0�\, \},
\end{equation}
with the stopping criterion
\begin{equation}
\Vert  \lambda_{h}^{k}- \lambda_{h}^{k-1}\Vert_{-1,h}\leq TOL.
\end{equation}

For the lowest order method, with 
\begin{equation} V_h = 
\{ v_h\in V  : v_h\vert _K \in P_1(K) \ \forall K \in \Ch \},  
\end{equation}
the Nitsche's method reduces to
\begin{equation}
   (\nabla u_h, \nabla v_h)    + \alpha^{-1}( \h^{-2}u_h, v_h)_{\Ohc} 
  =(f,v_h)_{\Omega\setminus \Ohc} + \alpha^{-1} (\h^{-2}g,v_h)_{\Ohc}
,
 \end{equation} 
which (except for  $(f,v_h)_{\Omega\setminus \Ohc} $ instead of $ (f,v_h)_{\Omega}$) is the standard  penalty formulation, cf.~Scholtz~\cite{SCH84}. 

Note that the a posteriori estimate of Theorem \ref{apostthstab}  still holds when the reaction force is computed from 
\eqref{react}.

Our conclusion is that the stabilized method can be implemented in a straightforward way using the above Nitsche's formulation. In practice, one can  replace $f$ and $g$ in \eqref{itlambda} with their  interpolants in $V_{h}$. 

\section{Numerical results}

Let $\Omega = \{(x,y) \in \mathbb{R}^2 : x^2+y^2 < 4\}$ and consider problem \eqref{stronglagr} with 
\begin{equation}
    \label{eq:examples}
    \left\{
    \begin{aligned}
        f(x,y) &= -1, \\
        g(r) &= \begin{cases}
            \sqrt{1-r^2} & \text{if $r<0.9$,}\\
            c_1 r+c_2         & \text{otherwise,}
        \end{cases}
    \end{aligned}
    \right.
\end{equation}
where $r=\sqrt{x^2+y^2}$ is the distance from the origin and $c_1, c_2$ are chosen such that the obstacle is $C^1$ in the whole domain.

The radial symmetry reduces \eqref{stronglagr}  to
the ordinary differential equation
\begin{equation}
    \label{eq:ode}
    \frac{\partial^2 u}{\partial r^2} + \frac{1}{r} \frac{\partial u}{\partial r} = 1, \ \ a< r< 2, \quad
    u(2)=0, \quad
    u(a)=g(a), \quad
    u'(a)=g'(a),
\end{equation}
where the unknowns are the function $u=u(r)$ and the radius $a$ of the contact area. Evidently $\lambda=0$ for $r>a$, and when $r<a$ the solution $(u,\lambda)$
satisfies
\begin{equation}
    u=g,\quad \text{and} \quad \lambda=1-\Delta g.
\end{equation}
Solving \eqref{eq:ode} leads to
\begin{align*}
    u(r) &= \left(1+a \log \frac{r}{2} \right)\left(\frac{r^2}{4}-1\right)+g'(r) a \log \frac{r}{2}, \quad \text{and} \quad a \approx 0.829\,.
\end{align*}
The solution $u$ has a step discontinuity in the second derivative in radial direction. Hence, it is globally only in $H^{5/2-\varepsilon}(\Omega), \ \varepsilon >0,$ but smooth in both the contact subregion and its complement.

First, the prescribed problem is solved by  mixed and stabilized methods using two mesh families: one that
follows the boundary of the true contact region (a conforming family of meshes) and an arbitrary mesh (nonconforming mesh)---see Fig.~\ref{fig:apriorimeshes}
for examples of the two different types of meshes.  In  Fig.~\ref{fig:analdisc}, the analytical solution is  compared against the discrete solutions  obtained by the  $P_2-P_0$ and $P_1-P_0$   stabilized methods. Note that the $P_1-P_0$ method does not (even for the conforming mesh) yield a reaction force converging in $L^{2}$. For the displacement we only give one picture for each mesh type since both methods give similar results.

The mesh is   refined uniformly and the errors of the displacement $u$  in the $H^1$-norm  and of the Lagrange multiplier $\lambda$ in the discrete $H^{-1}$-norm 
are computed and tabulated. The resulting convergence curves are visualized in Fig.~\ref{fig:aprioriuH1}
 as a function of the mesh parameter $h = \max_{K \in \mathcal{C}_h} h_K$. The parameter  $p$ stands for the  rate of convergence  $\mathcal{O}(h^{p})$.
The stabilization parameters were chosen through trial-and-error as $\alpha=0.1$ and $\alpha=0.01$ for the  $P_2-P_0$ and  $P_1-P_0$ stabilized methods, respectively.

The numerical example reveals that the limited regularity of the solution due to the unknown contact boundary
limits the convergence rate to $O(h^{3/2})$ and that the $H^1$-error of the lowest order methods are not affected by it.
When comparing the convergence rates of the Lagrange multiplier it can be seen that the lowest order mixed method
does not initially perform as well as the lowest order stabilized method. Through local elimination of the bubble
functions (cf. Remark \ref{elimination} above) this can be traced to a smaller effective stabilization parameter causing a larger constant in the
a priori estimate. If the stabilization parameter of $P_1-P_0$ method is further decreased, the performance
of the two methods will be identical.

It is further investigated whether the limited convergence rate due to the unknown contact boundary
can be improved by an adaptive refinement strategy. Based on the a posteriori estimate of the stabilized
method we define an elementwise error estimator as follows:
\begin{align*}
    \mathcal{E}_K(u_h,\lambda_h)^2&=h_K^2 \|\Delta u_h+\lambda_h+f\|_{0,K}^2 + \frac 12   h_{K} \| \llbracket \nabla u_h \cdot \boldsymbol{n} \rrbracket \|_{0,\partial K}^2 \\
              &\phantom{=}+\|(g-u_h)_+\|_{1,K}^2 + \int_K (u_h-g)_+\lambda_h\,\mathrm{d}x.
\end{align*}
Refining the triangles with $90\%$ of the total error
we create an improved sequence of meshes. See Fig.~\ref{fig:aposteriorimeshes} for examples of the resulting meshes.
We repeatedly adaptively refine the mesh and compute the solution and error of $P_2-P_0$ stabilized method. The resulting
convergence rates with respect to the number of degrees of freedom are given in Fig.~\ref{fig:aposteriorierror}. Note that the for a uniform mesh the relationship between the number of degrees of freedom and the mesh parameter is $N\sim h^{-2} $. Hence, we see that by the adaptivity we regain the optimal rate of convergence with respect to the degrees of freedom for the $P_2-P_0$ methods.
 
\section{Summary}

We have introduced families of bubble-enriched mixed and residual-based stabilized  finite element methods for discretizing the Lagrange multiplier formulation of the obstacle problem. We have shown 
that all methods yield stable approximations and proven the respective a priori and a posteriori error estimates. The lowest order methods have been tested numerically against an analytical solution and shown to lead to convergent solution strategies with optimal convergent rates.

\begin{figure}
    \centering
    \includegraphics[width=0.45\textwidth]{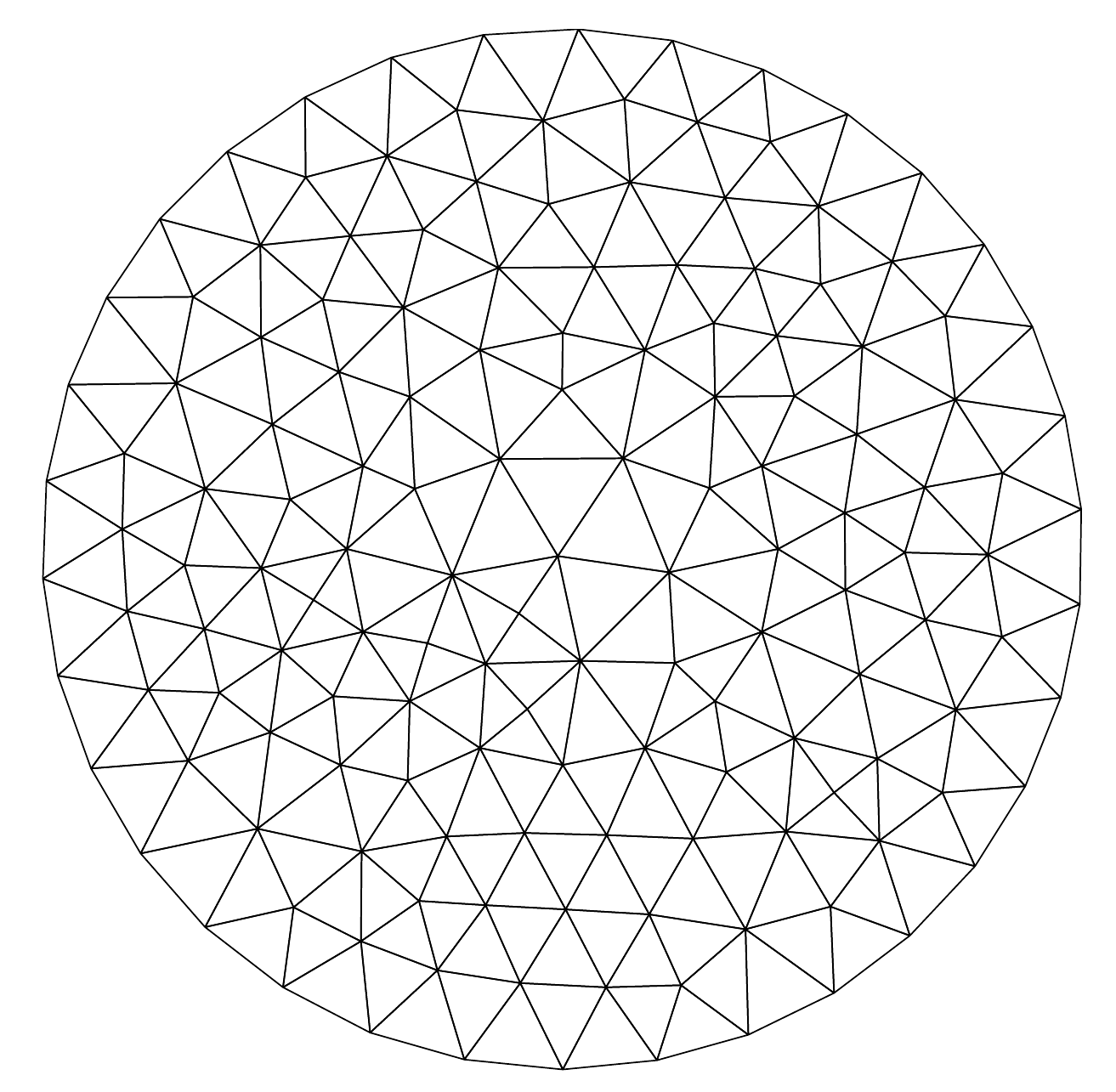}
    \includegraphics[width=0.45\textwidth]{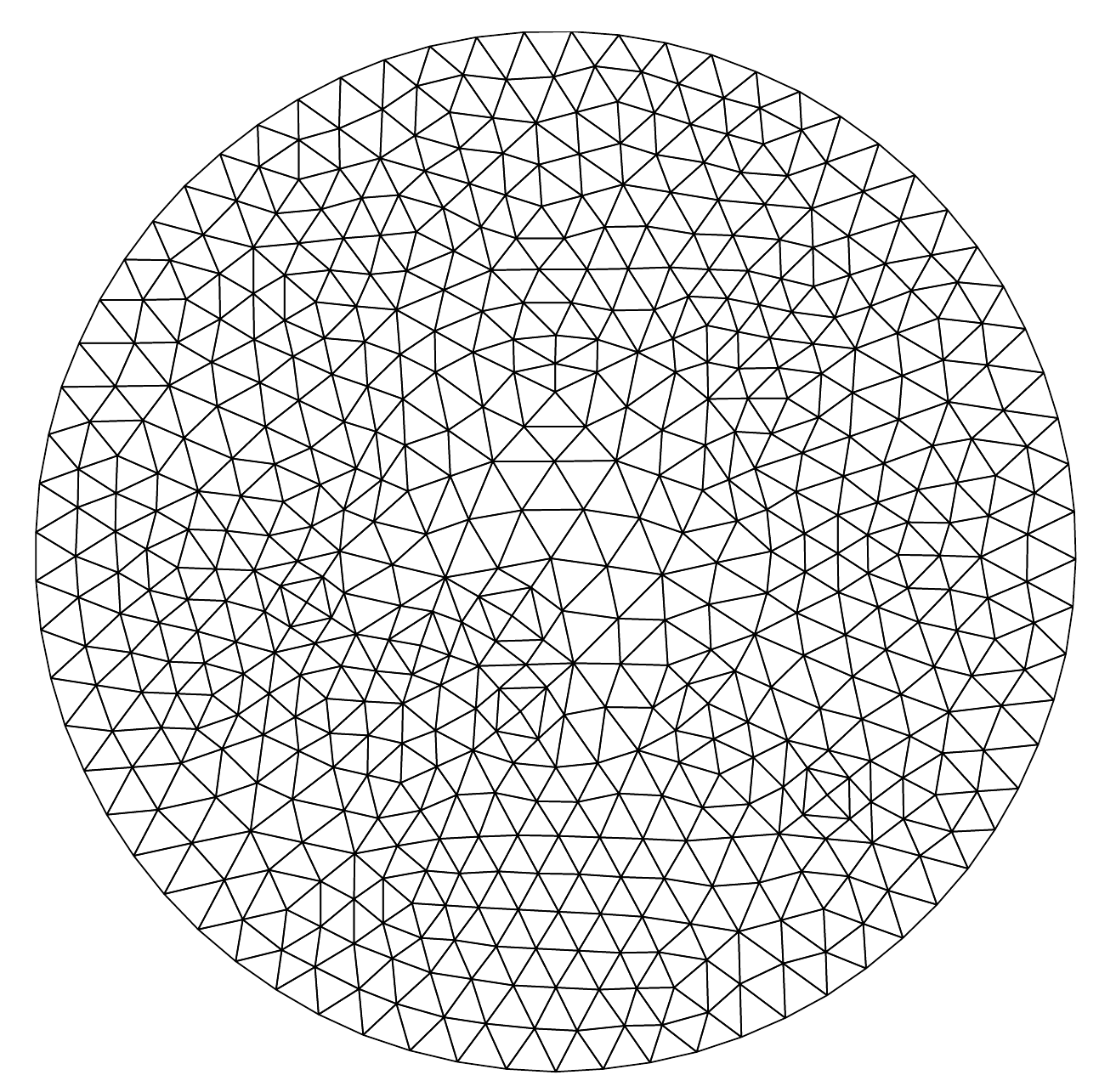}\\
    \includegraphics[width=0.45\textwidth]{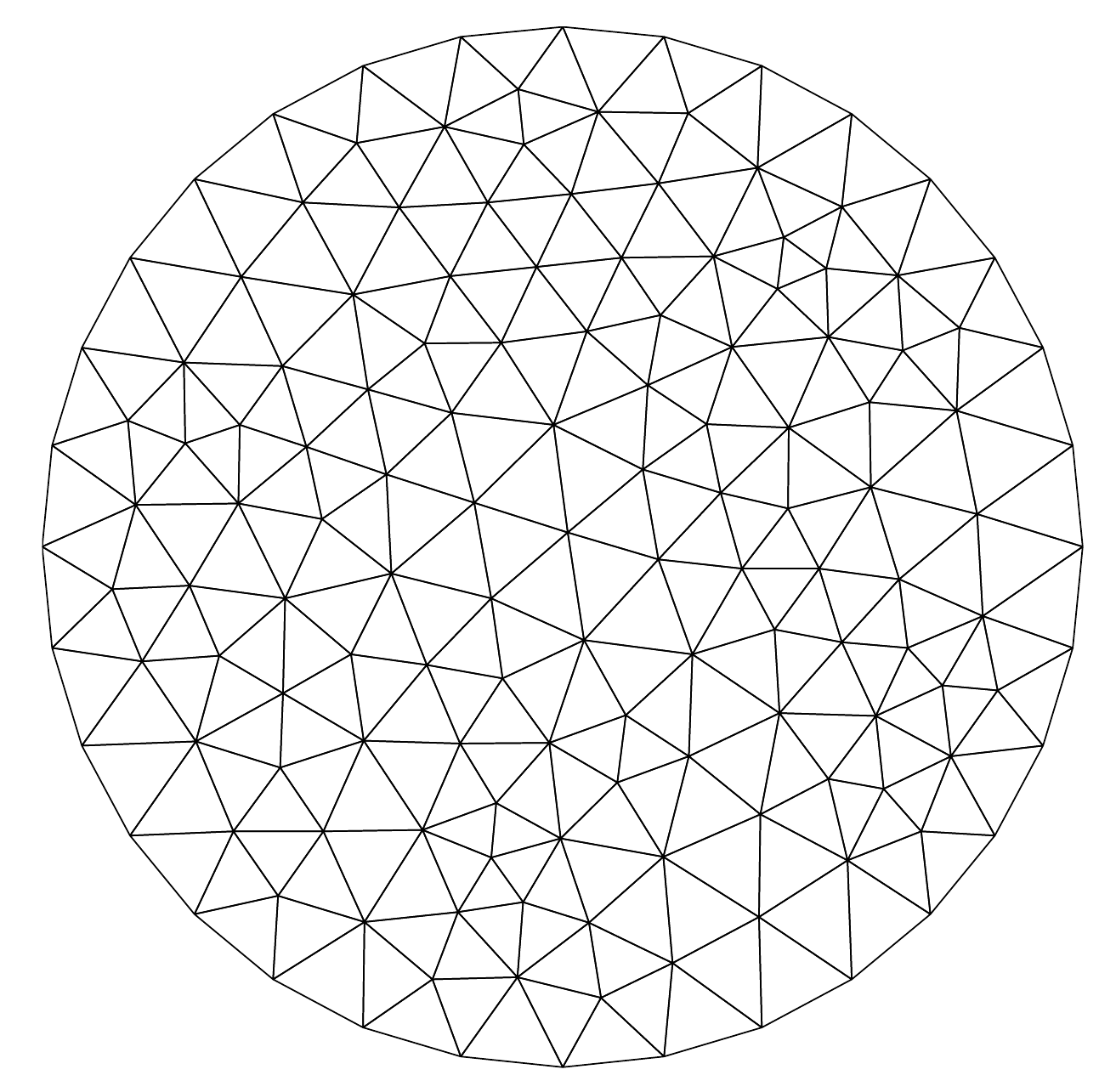}
    \includegraphics[width=0.45\textwidth]{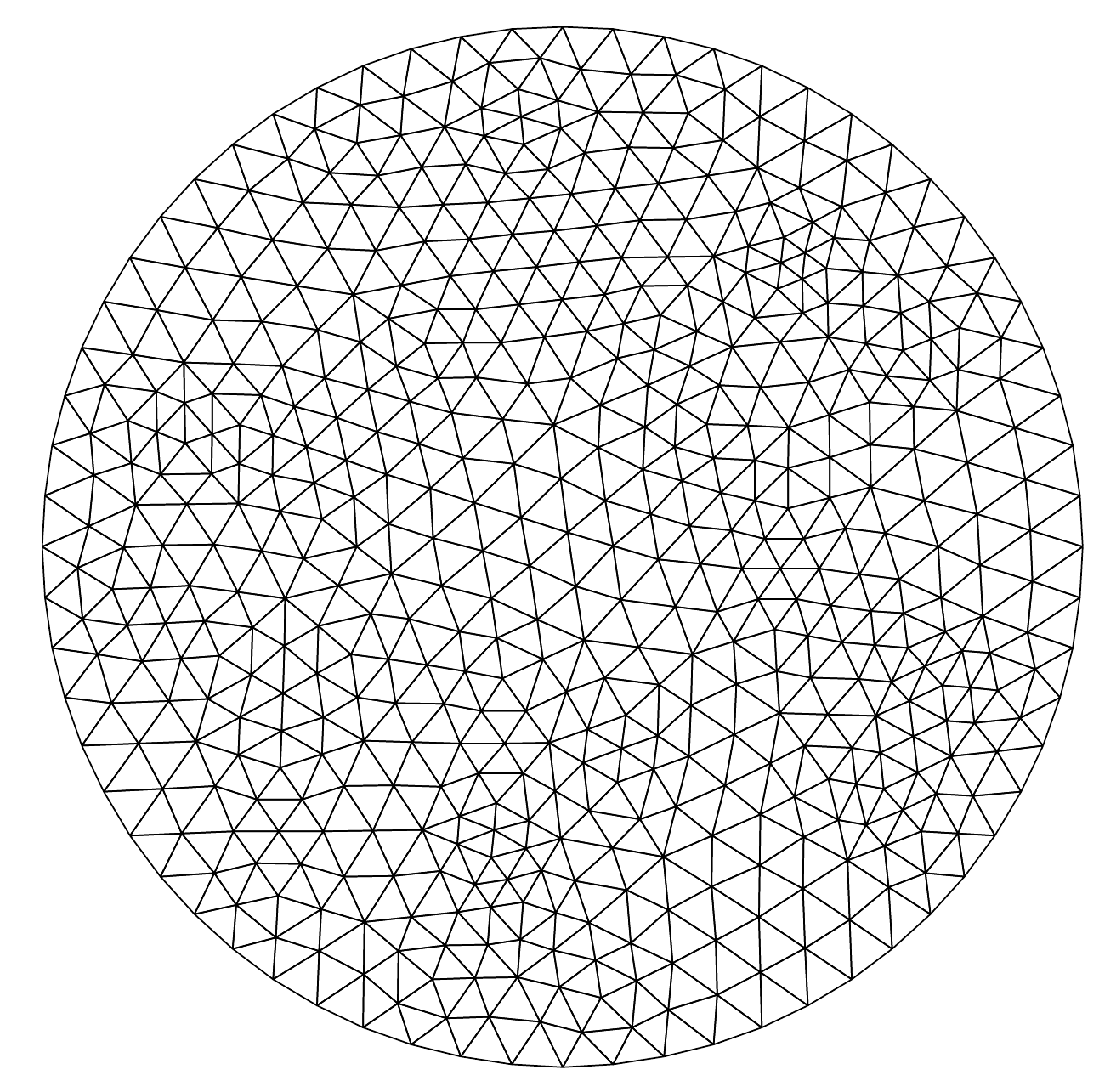}\\
    \caption{The two different mesh families: first conforming to the true contact boundary (upper panel), the other a general nonconforming family (lower panel).}
    \label{fig:apriorimeshes}
\end{figure}

\begin{figure}
    \centering
    \includegraphics[width=0.48\textwidth]{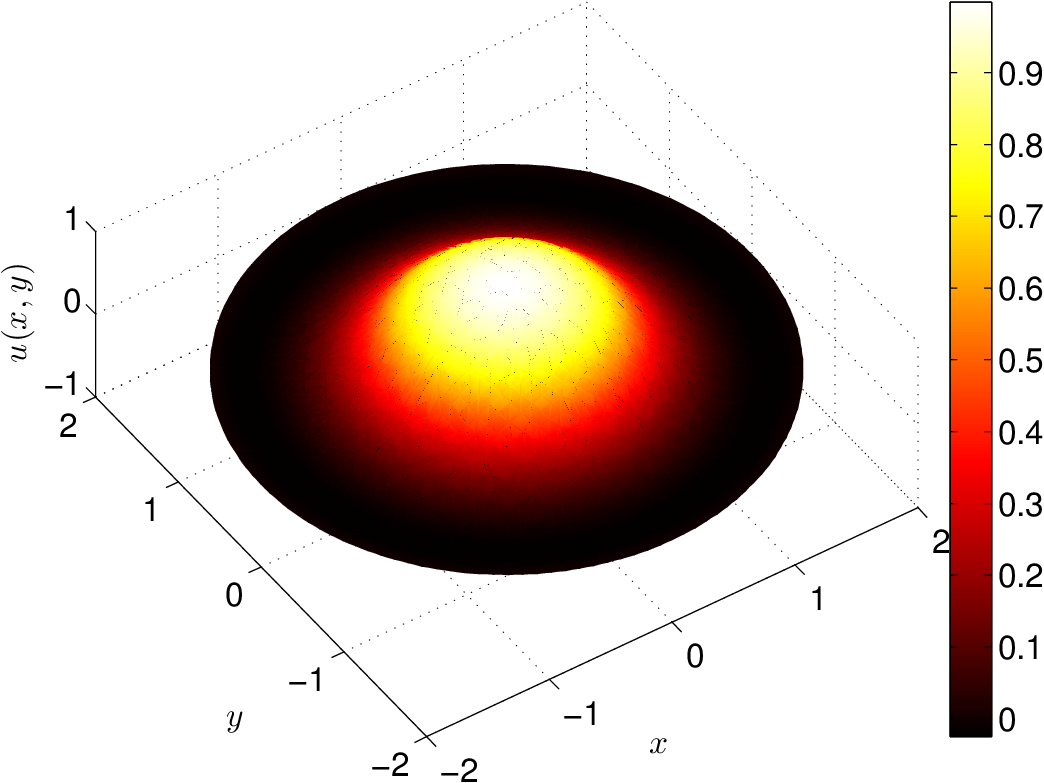}
    \includegraphics[width=0.48\textwidth]{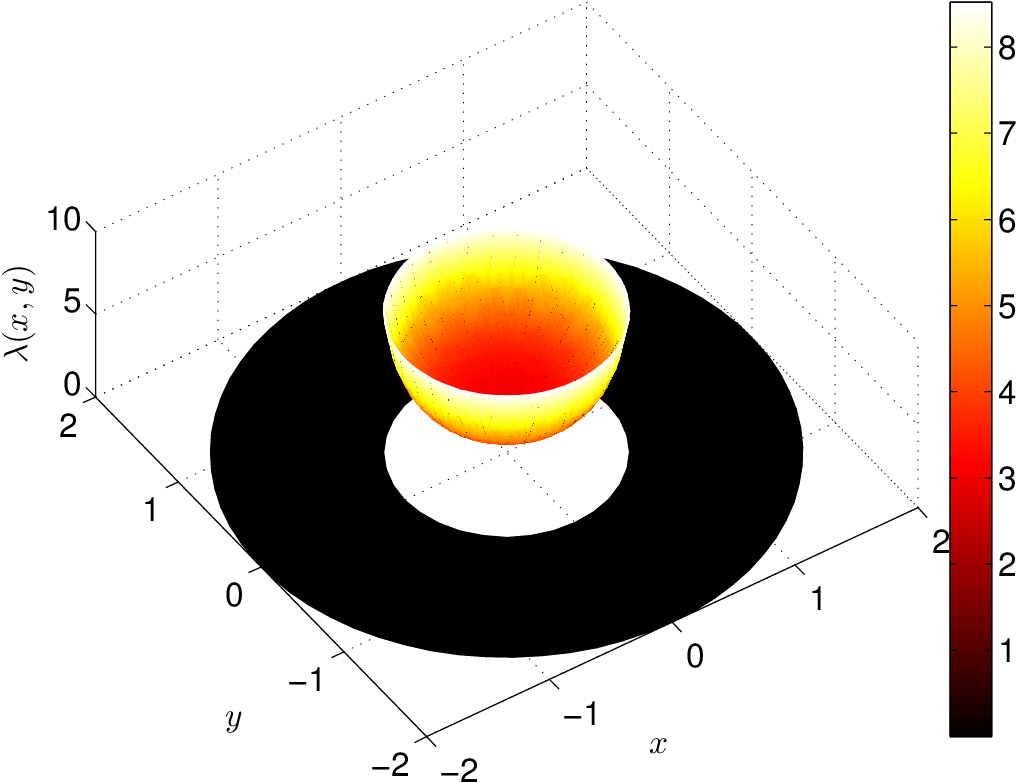}
    \includegraphics[width=0.48\textwidth]{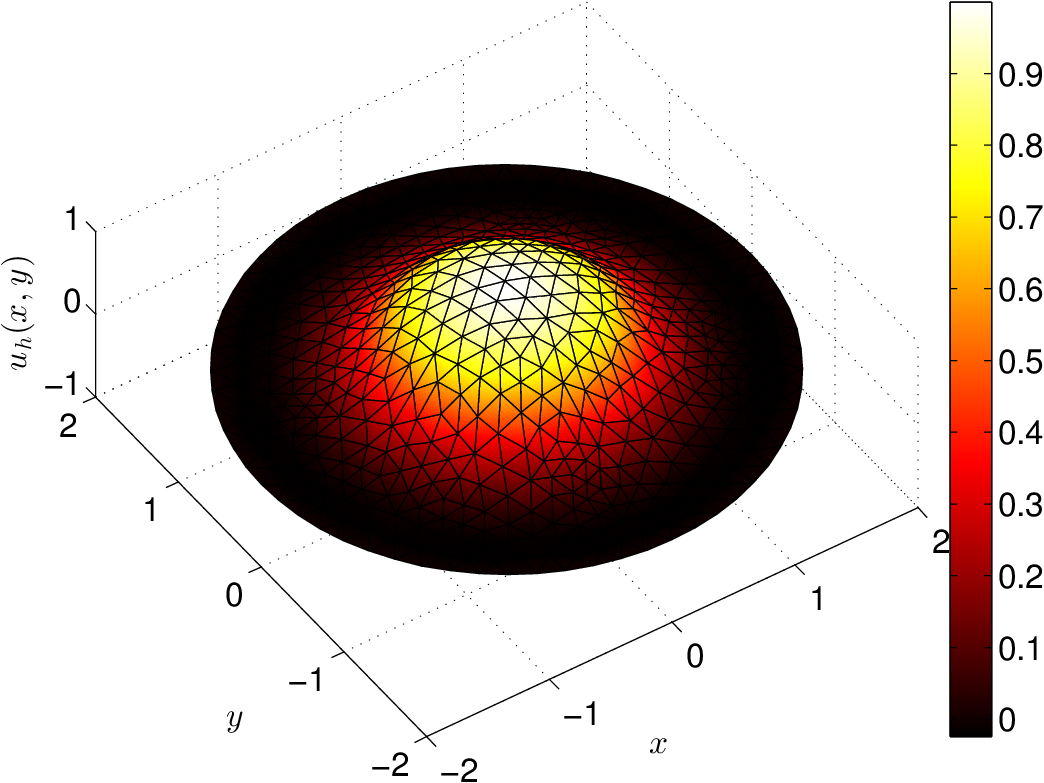}
    \includegraphics[width=0.48\textwidth]{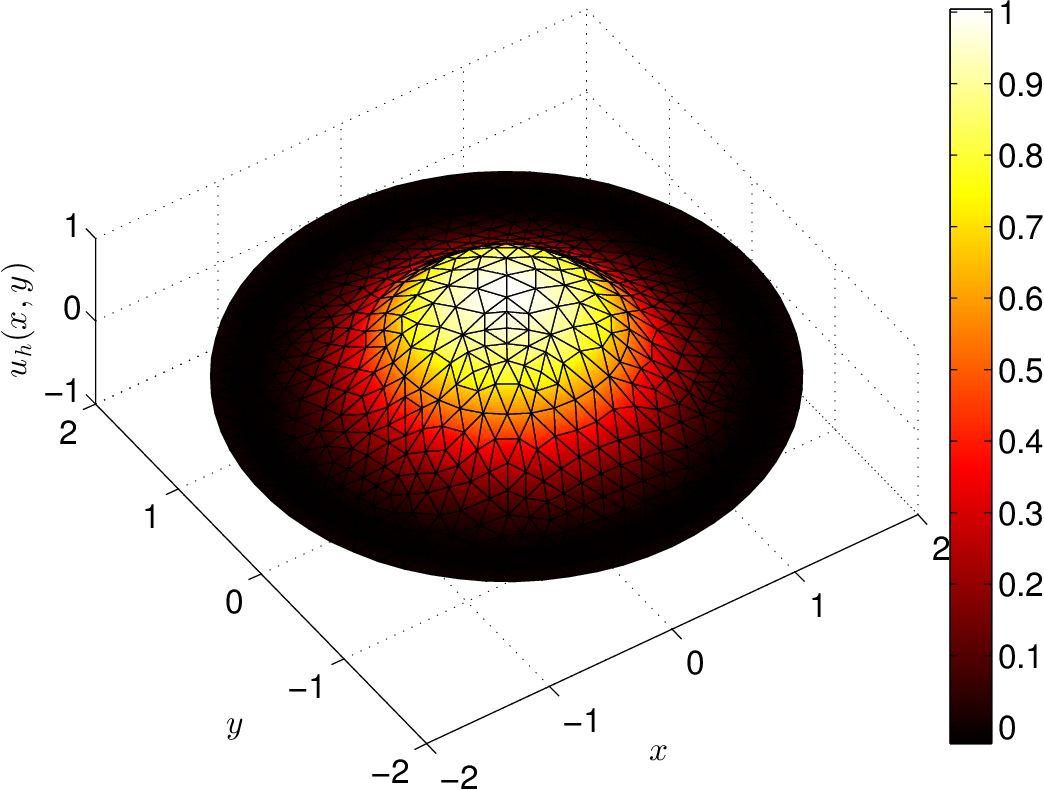}
    \includegraphics[width=0.48\textwidth]{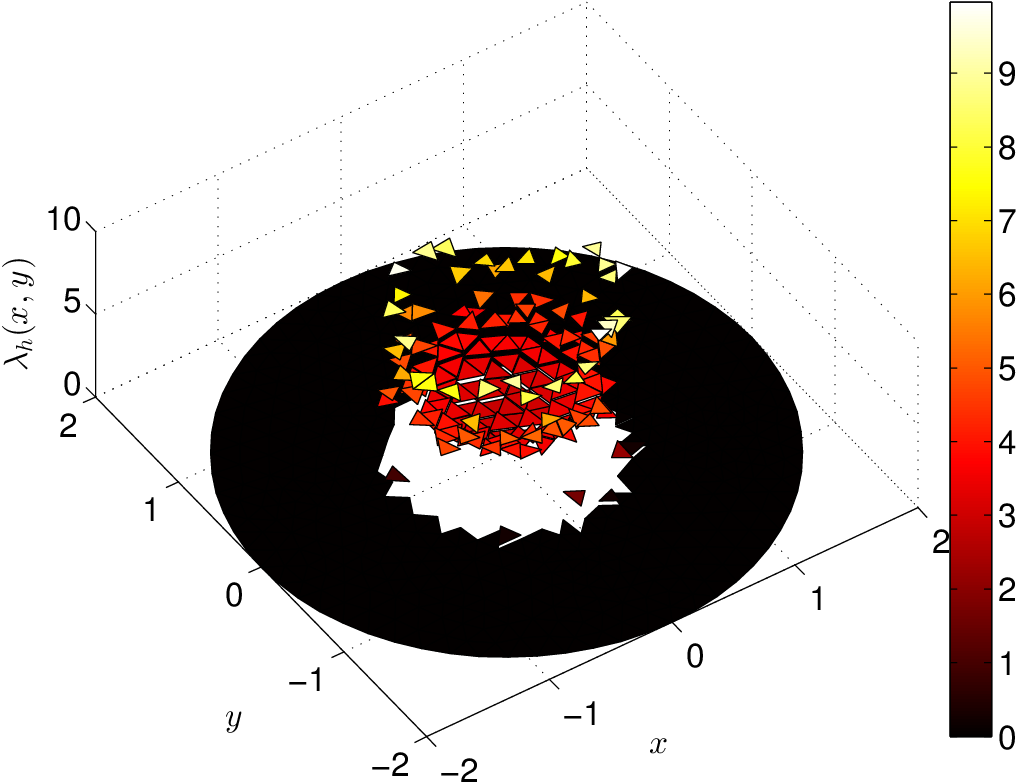}
    \includegraphics[width=0.48\textwidth]{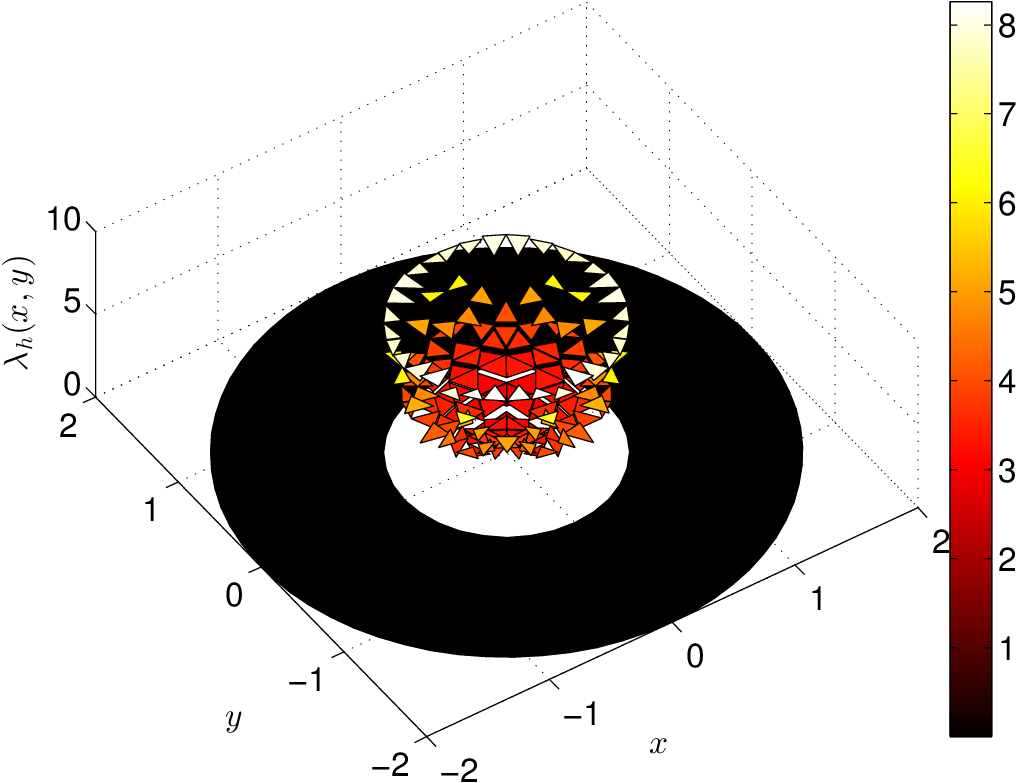}
    \includegraphics[width=0.48\textwidth]{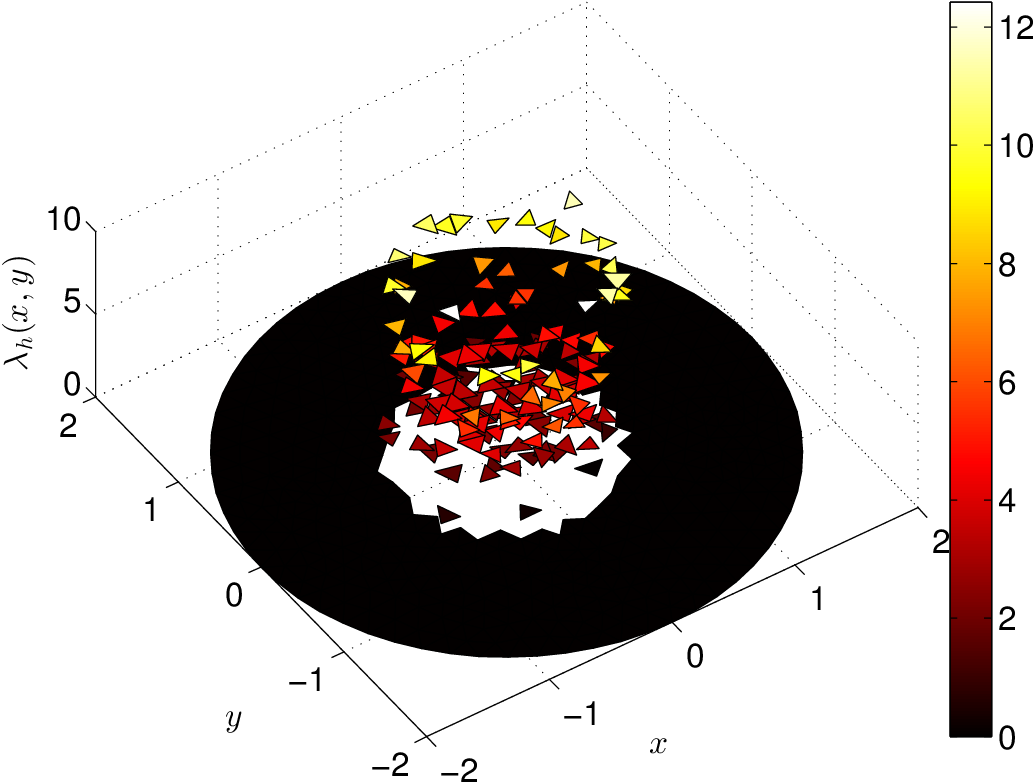}
    \includegraphics[width=0.48\textwidth]{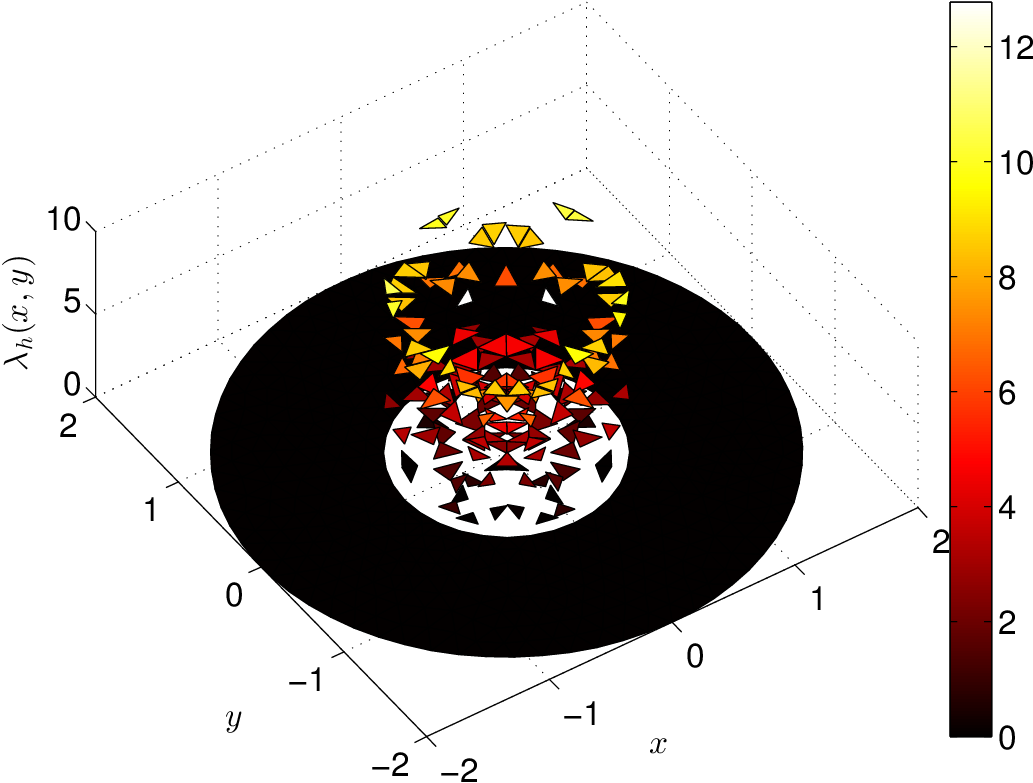}
    \caption{A comparison of analytic (upmost panel) and discrete (lower panels) solutions. The discrete solution was computed using nonconforming (left panel) and conforming (right panel) meshes and stabilized $P_2-P_0$ (third row) and $P_1-P_0$ (fourth row) methods.}
    \label{fig:analdisc}
\end{figure}

\begin{figure}
    \centering
    \includegraphics[width=\textwidth]{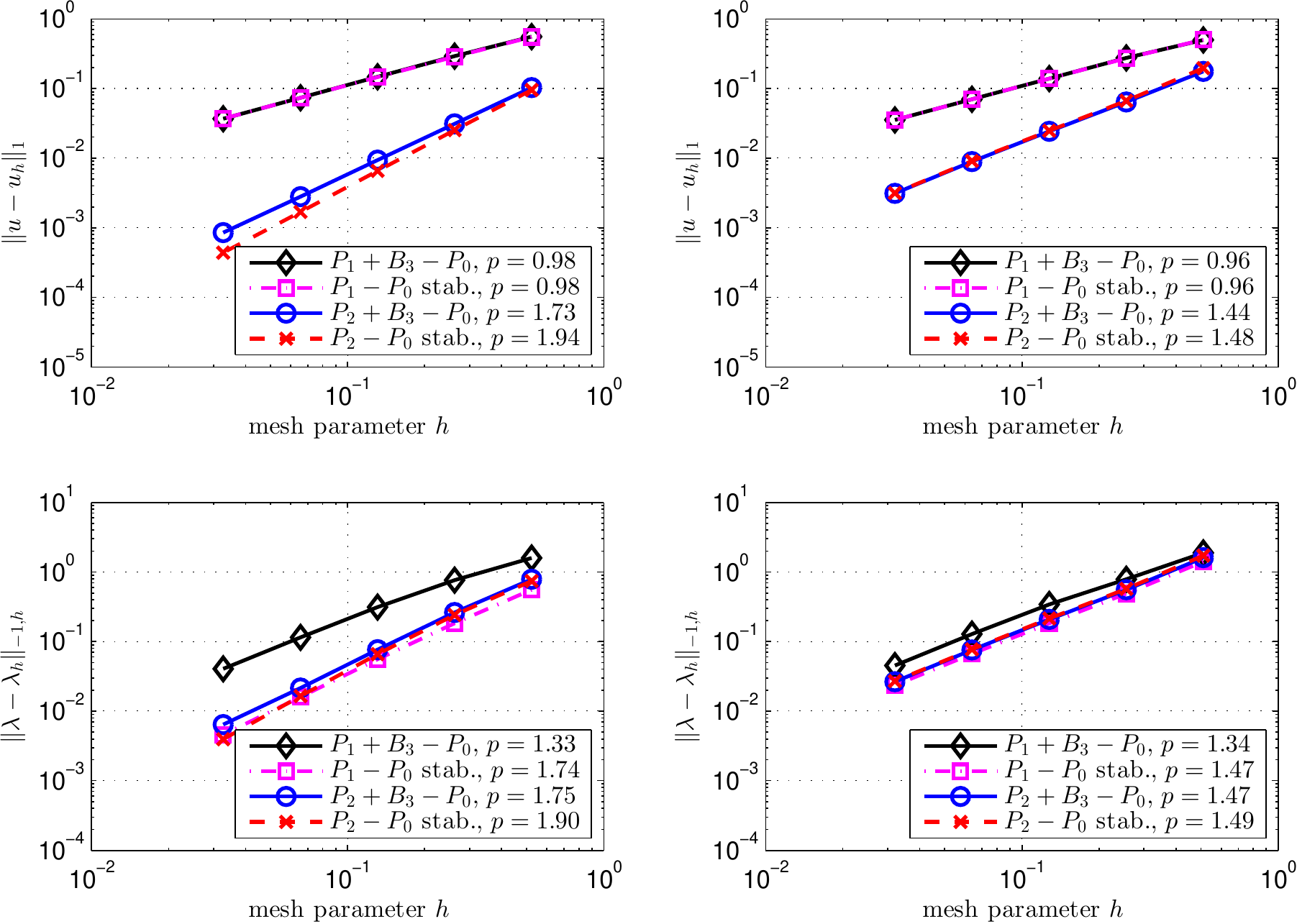}
    \caption{The convergence of the error $u-u_h$ in $H^1$-norm (upper panel) and the error $\lambda-\lambda_h$ in the discrete negative norm (lower panel) for the two different mesh families: conforming (left panel) and non-conforming (right panel).}
    \label{fig:aprioriuH1}
\end{figure}

\begin{figure}
   \centering
   \includegraphics[width=0.49\textwidth]{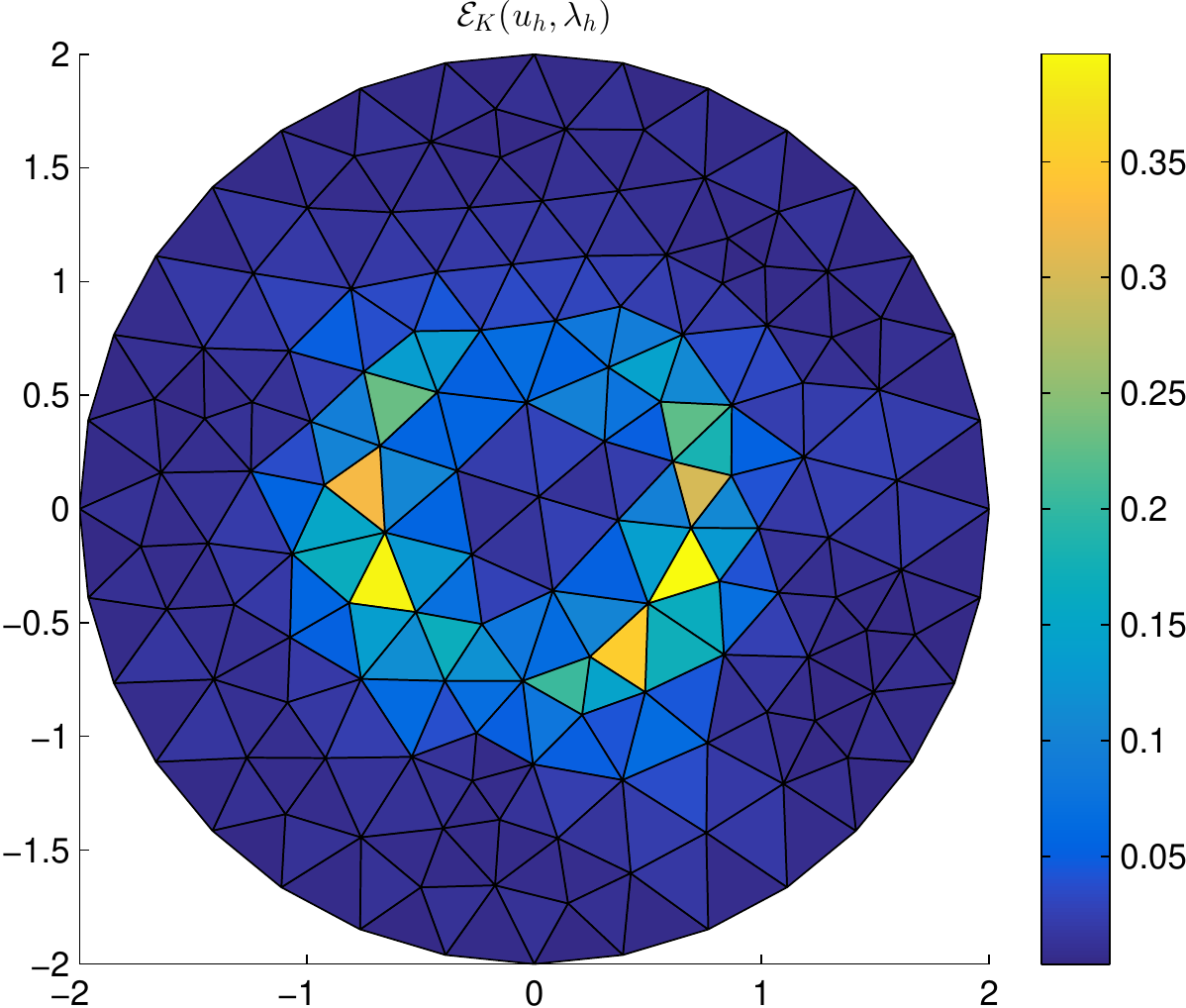}
   \includegraphics[width=0.49\textwidth]{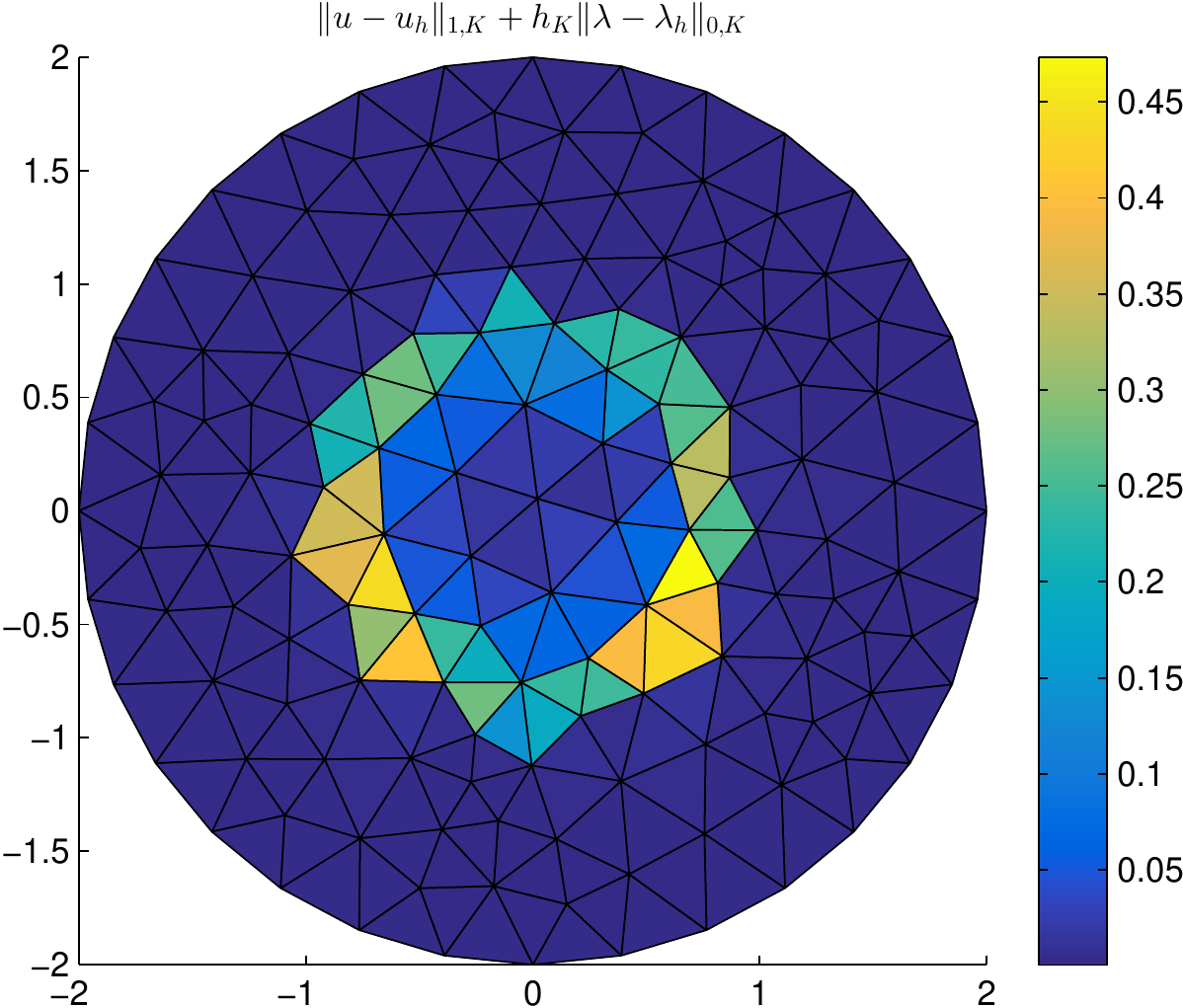}\\\vspace{0.25cm}
   \includegraphics[width=0.49\textwidth]{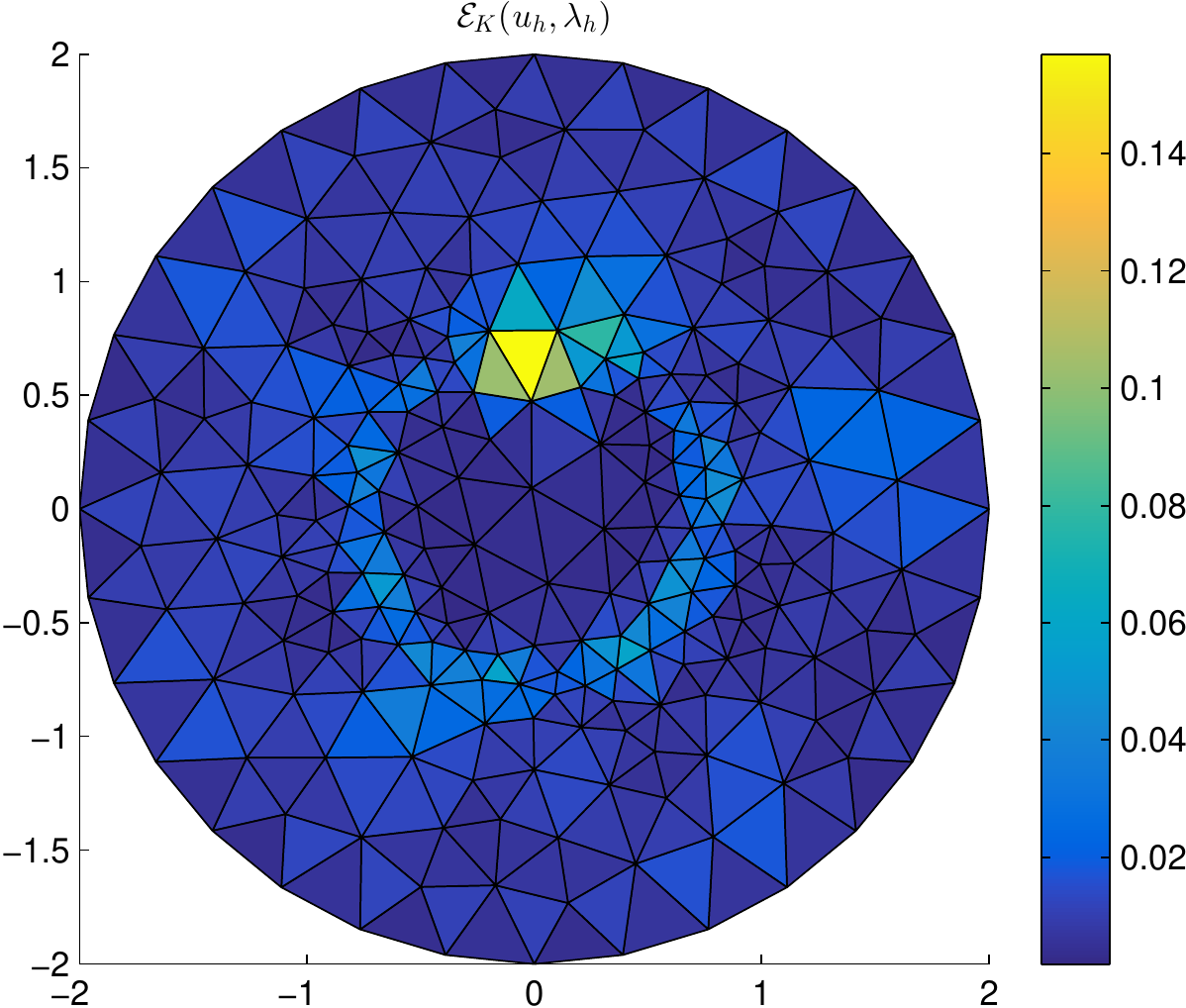}
   \includegraphics[width=0.49\textwidth]{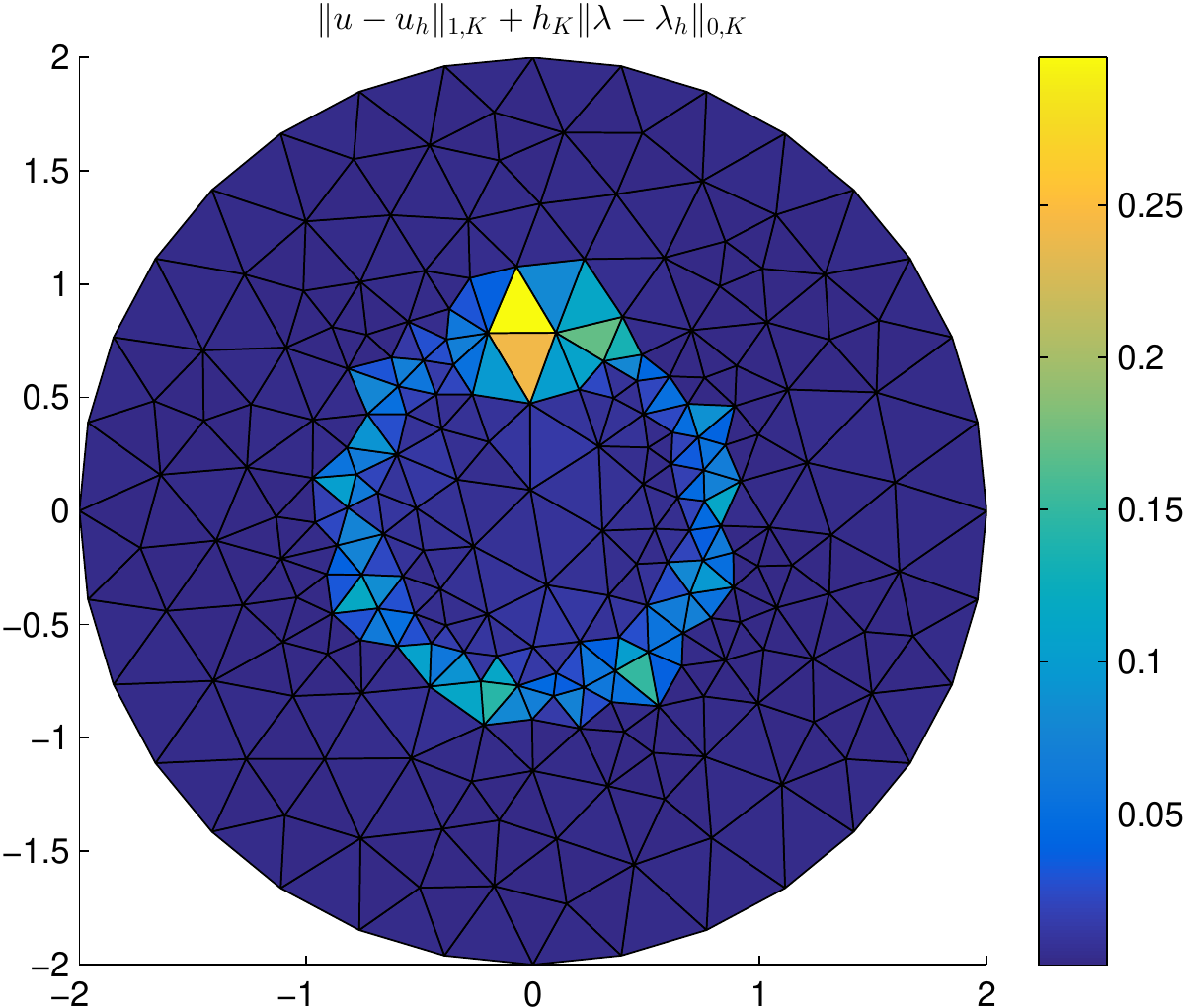}\\\vspace{0.25cm}
   \includegraphics[width=0.49\textwidth]{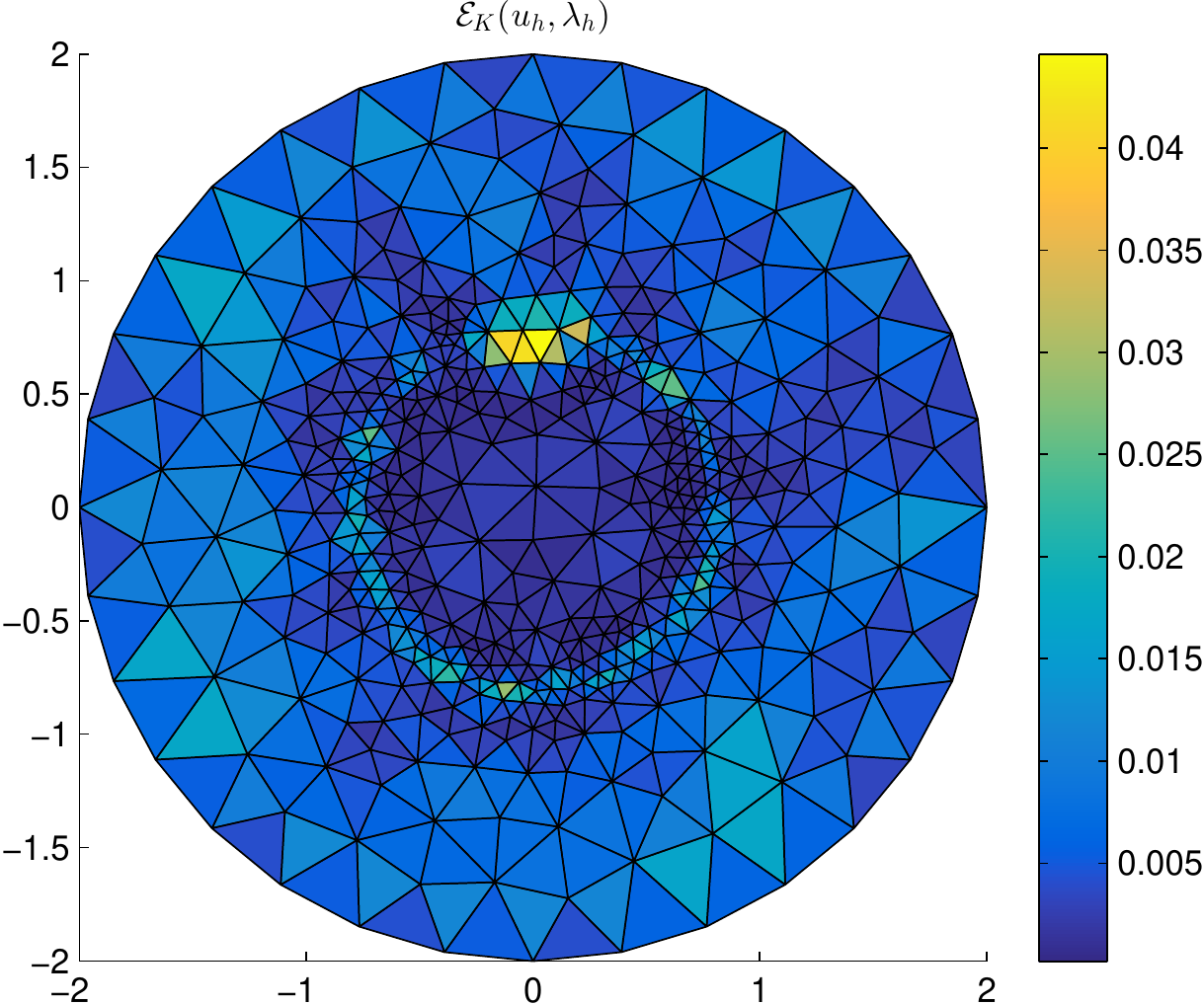}
   \includegraphics[width=0.49\textwidth]{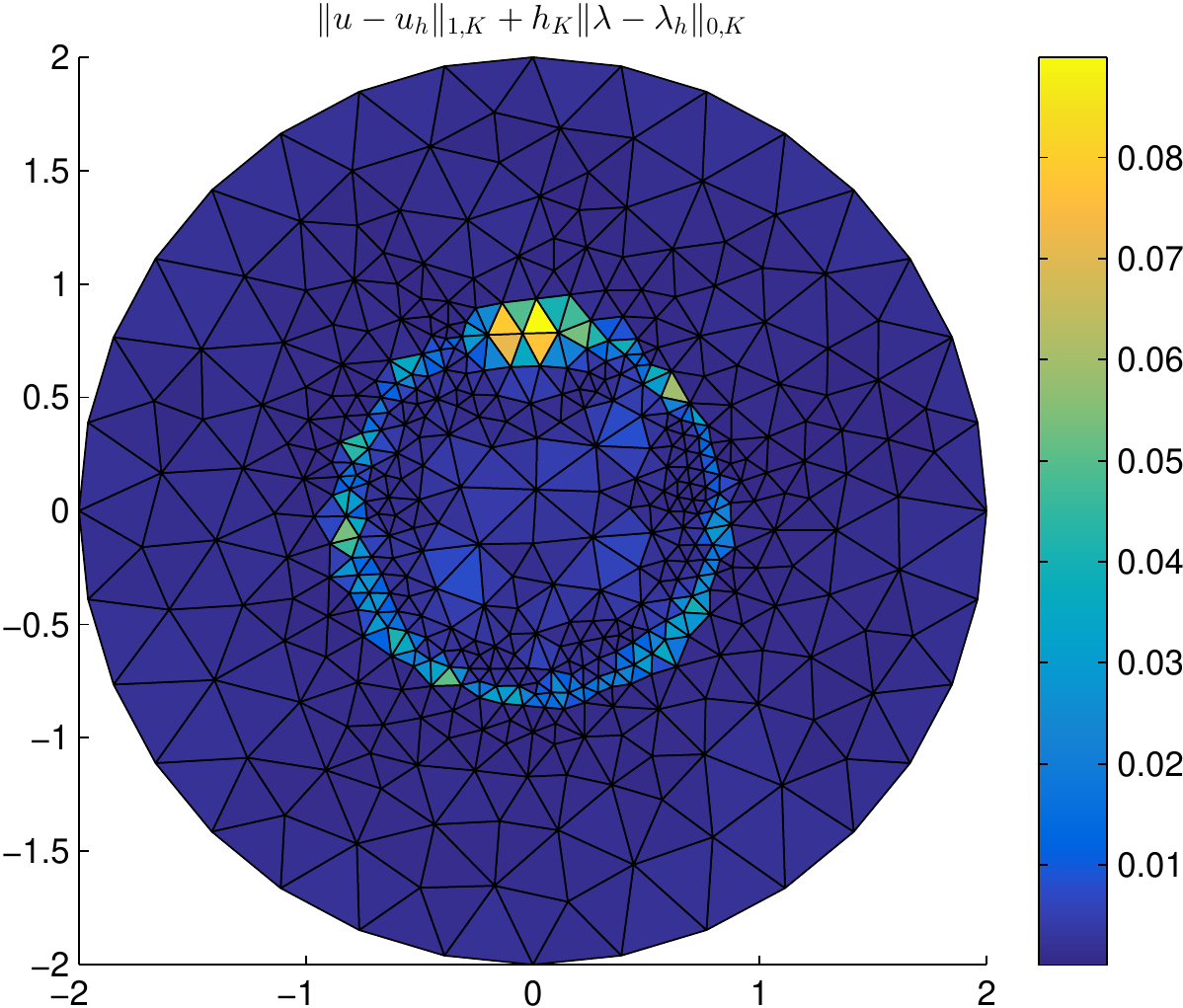}
   \caption{Three meshes arising from the adaptive refinement strategy. The local error estimators (left panels) are compared to the true local error (right 
   panels).}
   \label{fig:aposteriorimeshes}
\end{figure}

\begin{figure}
    \centering
    \includegraphics[width=\textwidth]{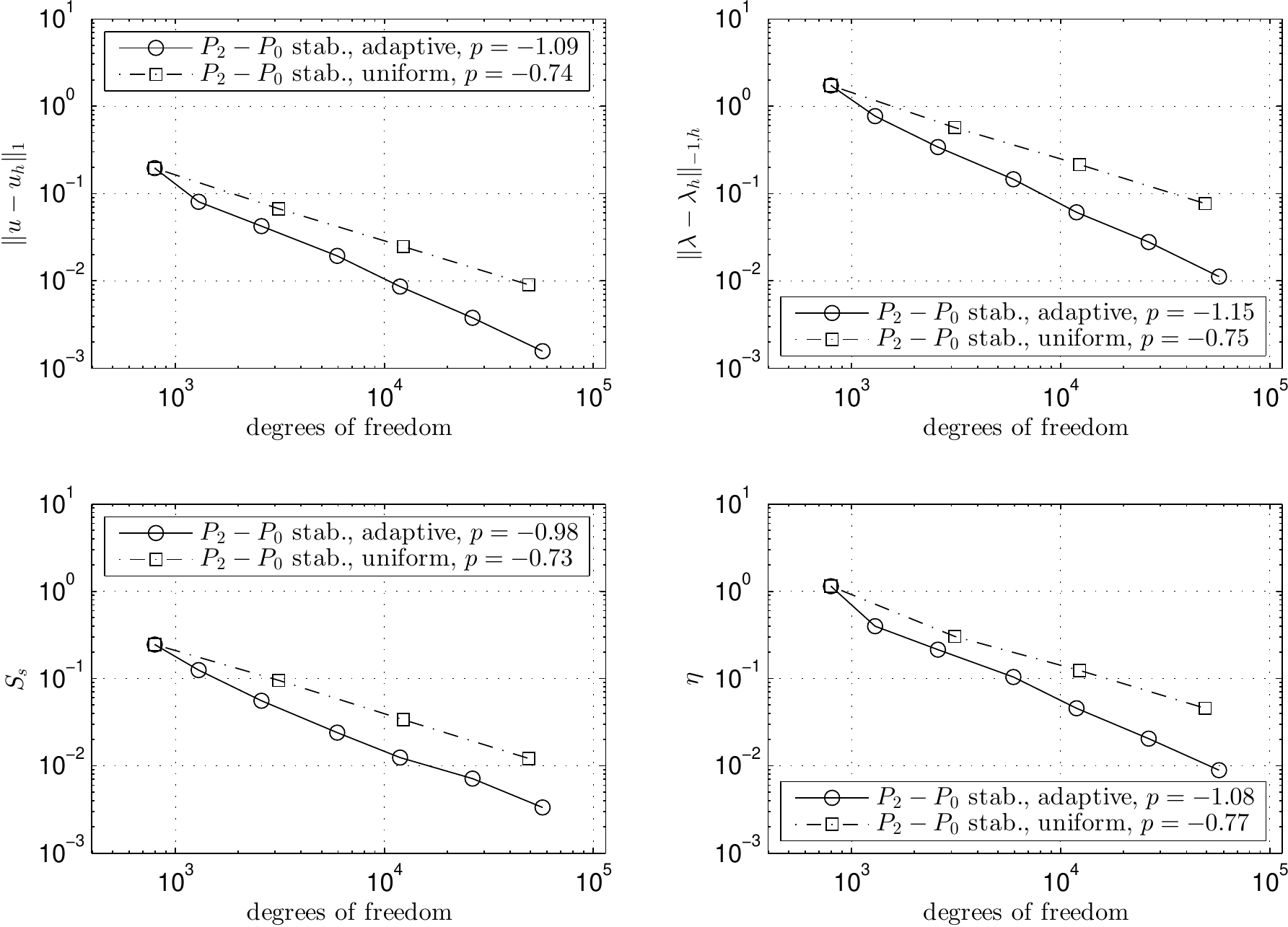}
    \caption{The convergence of the $P_2-P_0$ stabilized method with uniform and adaptive (nonconforming) refinements. The behavior of the error estimators $\eta$ and $S_s$ is shown separately for both cases.}
    \label{fig:aposteriorierror}
\end{figure}

\bigskip

{\bf Acknowledgements.}
Funding from Tekes -- the Finnish Funding Agency for Innovation (Decision number 3305/31/2015) and the Finnish Cultural
    Foundation is gratefully acknowledged.

\end{document}